\documentclass[11pt]{amsart}

\usepackage{amssymb,latexsym}
\usepackage{amsthm}
\usepackage{amsmath}
\usepackage{amsfonts}
\usepackage{graphicx}
\usepackage[all]{xy}
\usepackage{fancyhdr}
\usepackage[top=1in,bottom=1in,left=1.25in,right=1.25in]{geometry}

\newtheorem{theorem}{Theorem}[section]
\newtheorem{definition}[theorem]{Definition}
\newtheorem{remark}[theorem]{Remark}
\newtheorem{assumption}[theorem]{Assumption}
\newtheorem{proposition}[theorem]{Proposition}
\newtheorem{corollary}[theorem]{Corollary}

\def\Q{\mathbb{Q}}

\def\R{\mathbb{R}}
\def\Z{\mathbb{Z}}
\def\A{\mathbb{A}}

\def\Cm{\mathbb{C}}
\def\G{\mathbb{G}}

\def\T{\mathcal{T}}
\def\W{\mathcal{W}}
\def\IW{\mathcal{IW}}

\def\Fm{\mathcal{F}}
\def\vp{\varpi}

\def\lan{\langle}
\def\ran{\rangle}

\def\ra{\rightarrow}
\def\ov{\overline}
\def\ul{\underline}
\def\wh{\widehat}
\def\wt{\widetilde}
\def\st{\stackrel}
\def\tr{\textrm}

\usepackage[pdftex]{hyperref}

\begin{document}
\title{$p$-adic families of automorphic forms over some unitary Shimura varieties}
\author{Xu Shen}
\date{}
\address{Mathematisches Institut\\
Universit\"at Bonn\\
Endenicher Allee 60\\
53115, Bonn, Germany} \email{shen@math.uni-bonn.de}
\address{Current address: Fakult\"at f\"{u}r Mathematik\\
Universit\"at Regensburg\\
Universitaetsstr. 31\\
93040, Regensburg, Germany} \email{xu.shen@mathematik.uni-regensburg.de}
\renewcommand\thefootnote{}

\begin{abstract}
We construct some $n$-dimensional eigenvarieties for finite slope overconvergent eigenforms over some unitary Shimura varieties with signature $(1,n-1)\times(0,n)\times\cdots\times(0,n)$ by adapting Andreatta-Iovita-Pilloni's method. We also show that there are some Galois pseudo-characters over our eigenvarieties by studying analytic continuation of finite slope eigenforms over these Shimura varieties.
\end{abstract}

\maketitle
\tableofcontents

\section{Introduction}
The theory of $p$-adic families of automorphic forms started from the various works of Hida and Coleman. For the related history one can see the introduction of \cite{CM}. In loc. cit. Coleman and Mazur constructed the eigencurve for finite slope elliptic modular forms, which parameterizes the systems of eigenvalues for Hecke eigenforms. Since then, there have been many authors working on this field to develop a general theory of $p$-adic automorphic forms on higher rank groups. For example, there is the work of Urban \cite{U} based on studying overconvergent cohomology. In particular, Urban constructed eigenvarieties for any reductive groups $G$ over $\Q$ such that $G(\R)$ admits discrete series. There is also an approach of Emerton \cite{E} by studying the completed cohomology of arithmetic quotients. For definite unitary groups, Chenevier has constructed some eigenvarieties in \cite{Ch1} which have quite useful applications to Galois representations, see for example \cite{Ch2}. Finally, there is a geometric approach of Andreatta, Iovita and Pilloni \cite{AIP} for finite slope Siegel modular cuspforms. Their method is based on the theory of canonical subgroups, and is quite promising to be generalized to general PEL type Shimura varieties, see for example \cite{AIP2} and \cite{Br1}.

In this paper, we work out the construction of some $n$-dimensional eigenvarieties for finite slope eigenforms over some unitary Shimura varieties with signature $(1,n-1)\times(0,n)\times\cdots\times(0,n)$ by adapting Andreatta-Iovita-Pilloni's method in \cite{AIP}. These Shimura varieties are restricted to the class studied by Harris-Taylor in \cite{HT} for proving the local Langlands correspondence for $GL_n$. However, see Remark 6.5 for a discussion for more general Shimura varieties.

There were some related works in our setup. In \cite{Ka} Kassaei studied $p$-adic modular forms of integral weights over the Shimura curves, that is the case $n=2$. In \cite{Br} Brasca studied non integral weight forms over these Shimura curves and constructed an eigencurve. However, the definitions of modular forms in both \cite{Ka} and \cite{Br} can not be compared with the classical theory of automorphic forms, since their modular forms are sections of some line bundles which are not automorphic vector bundles over the Shimura curves. In \cite{D} Ding studied also non integral weight forms, whose definition is compatible with classical modular forms. He constructed in fact an eigenvariety of dimension two, and used it to study the local-global compatibility of Langlands correspondence in this setting. This paper deals with the higher dimension Shimura varieties.

As in \cite{D}, we study automorphic forms with fixed weights outside a fixed archimedean place $\tau_0$ (of the total real field $F^+$, see subsection 2.1). Let the weight $\kappa$ at $\tau_0$ vary in the $n$-dimensional $p$-adic weight space $\W$ (see subsection 2.2 for the definition), we construct some overconvergent automorphic vector bundles $\omega^{\dag\kappa}_w$ ($w\in\Q_{>0}$ is some rational number which depends on $\kappa$, see subsection 3.2) over (some admissible open subspaces inside the $p$-adic rigid analytification of) our Shimura varieties $X$ by applying the method in \cite{AIP}. These sheaves can be put in families when $\kappa$ varies in $\W$. In particular, we can construct $n$-dimensional eigenvarieties for the finite slope eigenforms over our Shimura varieties. An evident characteristic of the construction is that our eigenvarieties are partial, in the sense that we only let the weight at a fixed archimedean place vary. This is similar to the construction in \cite{Ch2}, but contrary to the total eigenvarieties as in \cite{Br1}.

An important assumption which we made in this paper is that the local reflex field is $\Q_p$. This is for having the density of the usual ordinary locus, cf. \cite{W} 1.6.3, so that we can apply the results on canonical subgroups of \cite{F2} Th\'eor\`eme 6 as in the way of Andreatta-Iovita-Pilloni. See Remark 6.5 for more explanation about the generalization.

We also study the analytic continuation of finite slope overconvergent eigenforms to prove that the classical points in our eigenvarieties are dense. As a consequence, there are some Galois pseudo-characters over the eigenvarieties. Let $L$ be a large extension of $\Q_p$ for the Shimura data, see page 4 of subsection 2.1 for more precise specification. The main theorem of the paper is the following.
\begin{theorem}
There is a rigid analytic variety $\mathcal{E}$ over $L$ and a locally finite map to the weight space $w:\mathcal{E}\ra\W$, such that
\begin{enumerate}
\item $\mathcal{E}$ is equidimensional of dimension $n$.
\item We have a character $\Theta: \mathbb{T}^{K^p}\otimes\mathbb{T}_p\ra \mathcal{O}(\mathcal{E})$. For any $\kappa\in\W$, $w^{-1}(\kappa)$ is in bijection with the eigensystems of $\mathbb{T}^{K^p}\otimes\mathbb{T}_p$ acting on the space of finite slope locally analytic overconvergent automorphic forms of weight $\kappa$.
\item For any $\kappa=(k_1,\dots,k_{n-1},k_n)\in \Z^{n-1}_+\times\Z\subset\W$, if $x\in w^{-1}(\kappa)$ satisfies  $v(\Theta_x(U_i))<k_{n-i}-k_{n-i+1}+1$ for $2\leq i\leq n-1$ and $v(\Theta_x(U_{1}))<k_n+k_{n-1}-n+1$, then the character $\Theta_x$ comes from a weight $\kappa$ automorphic eigenform on $X$. Here $\Theta_x$ is the composition of $\Theta$ with the evaluation map $ev_x: \mathcal{O}(\mathcal{E})\ra k(x)$ ($k(x)$ is the residue field of $x$).
\item There is a Galois pseudo-character $T: Gal(\ov{\Q}/F)\ra \mathcal{O}(\mathcal{E})^{red}$ ($F$ is some CM field, see subsection 2.1), such that for any point $x\in \mathcal{E}$, there is a continuous semi-simple representation $\rho_x: Gal(\ov{\Q}/F)\ra GL_n(\overline{k(x)})$ and the trace of this Galois representation is $T_x$. Here $\mathcal{O}(\mathcal{E})^{red}$ is the reduced algebra associated to $\mathcal{O}(\mathcal{E})$, and $T_x$ is the composition of $T$ with the evaluation map $ev_x: \mathcal{O}(\mathcal{E})^{red}\ra k(x)$.
\end{enumerate}
\end{theorem}
For the definition of the Hecke algebra $\mathbb{T}^{K^p}\otimes\mathbb{T}_p$ and the operators $U_i$ for $1\leq i\leq n-1$ see section 4. We hope that these eigenvarieties will have useful applications to Galois representations, as what those constructed by Chenevier for definite unitary groups have done.

This paper is organized as follows. In section 2, we introduce the related Shimura varieties and review automorphic vector bundles on them. In section 3, we first review the theory of canonical subgroups and the Hodge-Tate maps for them as in \cite{AIP} section 4, then we construct the overconvergent sheaves $\omega_w^{\dag\kappa}$ by proceeding in the same way as in loc. cit.. In section 4, we define the Hecke operators which act on the spaces of overconvergent automorphic forms. Then in section 5, we study analytic continuation of finite slope overconvergent eigenforms and prove the classicality theorem. In the last section we construct the $n$-dimensional eigenvarieties and prove the main theorem.\\
 \\
\textbf{Acknowledgments.} The tremendous debt that this paper owes to Andreatta-Iovita-Pilloni's work \cite{AIP} will be clear to the readers. I wish to acknowledge its serious impact on writing down this paper. I would like to thank Laurent Fargues sincerely for the discussion in the early stage of this work. I wish to thank Yiwen Ding, Peter Scholze, and Yichao Tian for useful discussions during the preparation of this paper. I would like to thank Fucheng Tan for useful comments. I should thank the referee for careful reading and useful suggestions. This work was supported by the SFB/TR 45 ``Periods, Moduli Spaces and Arithmetic of Algebraic Varieties'' of the DFG.

\section{Shimura varieties and automorphic vector bundles}
\subsection{Some unitary Shimura varieties}
In this section we introduce the Shimura varieties which will be the main object of study in this paper. These were studied by Harris-Taylor in \cite{HT} for proving the local Langlands correspondence for $GL_n$. For more details, see loc. cit. I.7, III.1, III.4.

Let $p$ be a prime number. Fix an imaginary quadratic field $E$ in which $p$ splits. The two primes of $E$ above $p$ will be denoted by $u$ and $u^c$, and the complex conjugation of $Gal(E/\Q)$ will be denoted by $c$. Let $F^+|\Q$ be a totally real field of degree $N$. Set $F=F^+E$, so that $F$ is a CM-field with maximal totally real subfield $F^+$. Let $\varpi=\varpi_1, \vp_2,\dots, \vp_r$ denote the primes of $F$ above $u$, and let $v=v_1,v_2,\dots,v_r$ denote their restrictions to $F^+$. We will denote the degrees of $F_{\vp_i}\simeq F^+_{v_i}$ by $d_i$ for $i=1,2,\dots,r$. Let $B/F$ denote a central division algebra of dimension $n^2$ over $F$ such that
\begin{itemize}
\item the opposite algebra $B^{op}$ is isomorphic to $B\otimes_{E,c}E$;
\item $B$ is split at $\vp$;
\item at any place $x$ of $F$ which is not split over $F^+$, $B_x$ is split (here and in the following $B_x=B\otimes F_x$);
\item at any place $x$ of $F$ which is split at $F^+$, either $B_x$ is split or $B_x$ is a division algebra;
\item if $n$ is even then $1+Nn/2$ is congruent modulo 2 to the number of places of $F^+$
above which $B$ is ramified.
\end{itemize}

As in \cite{HT} p.51, we can choose an involution of second kind $\ast$ on $B$. Moreover, we can choose some alternating pairing $\lan,\ran$ on $V\times V\ra \Q$ for the $B\otimes_FB^{op}$ module $V:=B$, which corresponds to another involution of second kind $\sharp$ on $B$. The associated reductive group $G/\Q$ is defined by
\[G(R)=\{(g,\lambda)\in (B^{op}\otimes_\Q R)^\times\times R^\times|\,gg^\sharp=\lambda\},\]
for any $\Q$-algebra $R$. Let $G_1$ be the kernel of the map $G\ra \mathbb{G}_m, (g,\lambda)\mapsto \lambda$, which can be viewed as a group over $F^+$. Choose a distinguished embedding $\tau_0: F^+\hookrightarrow \R$. As in Lemma I.7.1 of loc. cit. we can make the choice of the alternating pairing on $V\times V\ra \Q$ such that
\begin{itemize}
\item if $x$ is a rational prime which is not split in $E$, then $G$ is quasisplit at $x$,
\item if $\sigma: F^+\hookrightarrow \R$ is an embedding, then $G_1\times_{F^+,\tau}\R$ is isomorphic to the unitary group $U(1,n-1)$ if $\sigma=\tau_0$ and $U(n)$ otherwise.
\end{itemize}

We can say more about the group $G$. First, we have an isomorphism
\[Res_{E/\Q}G_E\simeq (B^{op}_\Q)^\times\times Res_{E/\Q}\mathbb{G}_m,\]which is an inner form of the quasi-split group $Res_{F/\Q}GL_n\times Res_{E/\Q}\mathbb{G}_m$. So the theory of automorphic representations for $G$ can be understood by those for $GL_n/F$ via the stable base change theorem of Clozel and Labesse, see 1.2.6 in \cite{H}. Second, the local reductive group at $p$ has the form
\[G_{\Q_p}\simeq \prod_{i=1}^r(B^{op}_{\vp_i})^\times \times \mathbb{G}_m.\]

Fix a maximal order $\Lambda_i=O_{B_{\vp_i}}$ in $B_{\vp_i}$ for each $i=1,2,\dots,r$. Our pairing on $V=B$ induces perfect duality between $V_{\vp_i}$ and $V_{\vp^c_i}$. Let $\Lambda^\vee \subset V_{\vp^c_i}$ be the dual of $\Lambda_i\subset V_{\vp_i}$. Then
\[\Lambda:=\bigoplus_{i=1}^r\Lambda_i\oplus\bigoplus_{i=1}^r\Lambda_i^\vee\] is a $\Z_p$-lattice in $V\otimes_\Q\Q_p$ and we have a perfect pairing $\Lambda\times\Lambda\ra\Q_p$. Let $O_B\subset B$ be the unique maximal $\Z_{(p)}$-order such that $O_B^\ast=O_B$ and $O_{B,\vp_i}=O_{B_{\vp_i}}$ for $i=1,2,\dots,r$. Fix an isomorphism $O_{B_\vp}\simeq M_n(O_{F_\vp})$. Let $\varepsilon=(\varepsilon_{ij})\in M_n(O_{F_\vp})$ be the idempotent with $\varepsilon_{11}=1$ and all the other $\varepsilon_{ij}=0$. Then $\Lambda_{11}:=\varepsilon\Lambda_1\simeq (O_{F_\vp}^n)^\vee$, and
\[\Lambda\simeq ((O_{F_\vp}^n\otimes\Lambda_{11})\oplus (O_{F_\vp}^n\otimes\Lambda_{11})^\vee)\oplus \bigoplus_{i=2}^r(\Lambda_i\oplus\Lambda_i^\vee).\]

Let $K\subset G(\A_f)$ be a sufficiently small open compact subgroup. Then we have a projective Shimura variety $Sh_K$ over $F$, which is a moduli space of abelian varieties with additional structures. More precisely, for any connected locally noetherian $F$-scheme $S$, $Sh_K(S)=\{(A,\lambda,\iota,\bar{\eta})\}/\simeq$ where
\begin{itemize}
\item $A/S$ is an abelian scheme of dimension $Nn^2$;
\item $\lambda: A\ra A^\vee$ is a polarization;
\item $\iota: B\ra End(A)\otimes\Q$ is an action such that $\lambda\circ\iota(b)=\iota(b^\ast)^\vee\circ\lambda$ for all $b\in B$ and $(A,\iota)$ is compatible (cf. \cite{HT} Lemma III.1.2);
\item $\bar{\eta}$ is a level structure $\eta: V\otimes\A_f\ra V_f(A)\,(mod\,K)$.
\end{itemize}
Let $L|\Q_p$ be a finite extension which is large enough so that it contains all the embeddings of $F$ into $\ov{\Q}_p$ and it splits all $B_{\vp_i}$ for $2\leq i\leq r$. We fix an embedding $F_\vp\subset L$ and still denote $Sh_K$ the above variety base changed to $L$ via the fixed embedding by abuse of notation. Assume $K$ has the form $K=K^pK_p\subset G(\A_f^p)\times G(\Q_p)$. If $K_p=\prod_{i=1}^rO_{B^{op}_{\vp_i}}^\times\times\Z_p^\times$, then one can define a proper smooth intergral model $S_K$ of $Sh_K$ over $O_L$ by considering a similar integral modular problem. In fact, let $m=(m_1,\dots,m_r)\in\Z^r_{\geq 0}$, and $K^p(m)$ be the product
\[K^p\times\prod_{i=1}^rker(O_{B_{\vp_i}^{op}}^\times\ra(O_{B_{\vp_i}^{op}} /\vp_i^{m_i})^\times)\times\Z^\times_p,\]
then by introducing the notion of Drinfeld level structures at $p$, one can define a proper flat regular model $S_{K^p(m)}$ of $Sh_{K^p(m)}$ over $O_L$. For $K^p$ and $m$ varying, the group $G(\A_f)$ acts as Hecke correspondences on the tower $(S_{K^p(m)})_{K^p,m}$. Over $\Cm$, $Sh_{K,\Cm}$ is a disjoint union of $ker^1(\Q,G)$ copies of the PEL unitary Shimura variety $Sh_K(G,X)$ associated to the corresponding Shimura data. By abuse of terminology, we will call these moduli spaces for $K$ varying Shimura varieties.

\subsection{Automorphic vector bundles}
Let $(m_2,\dots,m_r)\in \Z^{r-1}_{\geq 0}$ and $K^p$ be fixed. We will be interested in the levels $K_0=K^p(0,m_2,\dots,m_r)$ and \[K=K^p\times Iw\times\prod_{i=2}^rker(O_{B_{\vp_i}^{op}}^\times\ra(O_{B_{\vp_i}^{op}} /\vp_i^{m_i})^\times)\times\Z_p^\times\subset K^p(1,m_2,\dots,m_r),\] where $Iw\subset GL_n(O_{F_\vp})$ denotes the Iwahori subgroup. To simplify notations, let $X:=Sh_K, Y:=Sh_{K_0}$ and $\ov{X}$ (resp. $\ov{Y}$) be the special and generic fibers of $S_K$ (resp. $S_{K_0}$). We know that $S_{K_0}$ is smooth (\cite{HT} Lemma III.4.1), and $S_K$ has strictly semi-stable reduction (\cite{TY}, Proposition 3.4). We will study automorphic vector bundles on $X$ and the related theory of automorphic forms. For this we will need some preparation.

Let $\Sigma=Hom(F^+,L)$. Then it is bijective to $Hom(F^+,\R)$, after fixing an embedding $L\hookrightarrow \ov{\Q}_p$ and an isomorphism $\ov{\Q}_p\simeq \Cm$. Let $\tau\in\Sigma$ be the element which corresponds to $\tau_0$ under this bijection for the fixed $\tau_0: F^+\hookrightarrow\R$ in the last subsection. We can and we do take the embedding and isomorphism above such that $\tau$ comes from an element in $Hom(F_\vp, L)$.
Consider the Levi subgroup $M$ of \[G_{F_\vp}\simeq \G_m\times \prod_{\sigma:F_\vp\hookrightarrow L}GL_n\times\prod_{i=2}^r(B^{op}_{\vp_i})^\times_{F_\vp}\] defined by \[M=\G_m\times GL_{n-1}\times GL_1\times \prod_{\sigma\neq\tau:F_\vp\hookrightarrow L}GL_n\times\prod_{i=2}^r(B^{op}_{\vp_i})^\times_{F_\vp}.\]
Then over $L$, we have the isomorphisms
\[\begin{split}&G_L\simeq \G_m\times\prod_{\sigma\in\Sigma}GL_n\\
&M_L\simeq \G_m\times GL_{n-1}\times GL_1\times \prod_{\sigma\neq\tau: F^+\hookrightarrow L}GL_n.\end{split}\]Take a maximal torus $T\subset M_L\subset G_L$, and choose a Borel subgroup $T\subset B\subset G_L$. Then we have an isomorphism of dominant weight (of  $B\cap M_L$)
\[
X^\ast(T)^{M_L}_+\simeq\Z\times\Z^{n-1}_+\times\Z\times\prod_{\sigma\neq\tau}\Z^n_+.\] Here and in the following, for any positive integer $k$, $\Z^k_+=\{(a_1,\dots,a_k)\in \Z^k|\; a_1\geq \cdots\geq a_k\}$. Therefore, each irreducible representation $\rho$ of $M_L$ corresponds to  $(b_0,\overrightarrow{b_{\tau}},b_{\tau},(\overrightarrow{b_{\sigma}})_{\sigma\neq\tau})$, with  $\overrightarrow{b_{\tau}}=(b_{1\tau},\dots,b_{(n-1)\tau})\in\Z^{n-1}_+, \overrightarrow{b_{\sigma}}=(b_{1\sigma},\dots,b_{n\sigma})\in\Z^n_+$, and this gives arise to an automorphic vector bundle $\mathcal{V}_\rho$ on $X$.

Let $A/S_K$ be the universal abelian scheme with unit section $e: S_K\ra A$. Recall $S_K$ is the integral Shimura varieties over $O_L$ with generic fiber $X$. Let $\omega:=\varepsilon(e^\ast\Omega^1_{A/S_K})$, where $\varepsilon=(\varepsilon_i^+\oplus\varepsilon^-_i)$, for each $i=1,\dots,r$ both $\varepsilon_i^+$ and $\varepsilon_i^-$ are the idempotent matrices with all $(k,j)\neq(1,1)$-elements 0 and 1 for the (1,1) element. We have the decomposition
\[A[p^\infty]=A[\varpi_1^\infty]\oplus\cdots\oplus A[\varpi_r^\infty]\oplus A[\varpi_1^{c,\infty}]\oplus\cdots\oplus A[\varpi_r^{c,\infty}],\]with $A[\varpi_i^\infty]$ ind-\'etale for $2\leq i\leq r$, and $A[\varpi_i^\infty]^D\simeq A[\varpi_i^{c,\infty}]$ for all $1\leq i\leq r$. This induces a decomposition
\[\omega=\omega_1\oplus\cdots\oplus\omega_r\oplus\omega_1^c\oplus\cdots\oplus\omega_r^c,\]
with $\omega_i$ comes from the sheaf of invariant differentials of $A[\varpi_i^{\infty}]$ and $\omega_i^c$ comes from the sheaf of invariant differentials of $A[\varpi_i^{c,\infty}]$. This is also the decomposition of $\omega$ under the action of $O_F\otimes\Z_p$. Thus $\omega_2=\cdots=\omega_r=0$, rank$\omega_i^c=d_in, i\geq 2$, rank$\omega_1^c=d_1n-1$, rank$\omega_1=1$. For $1\leq i\leq r$, let $\Sigma_i=Hom(F_{\vp_i},L)$. From the next subsection we will assume $F_{\vp_1}=\Q_p$. Then $\Sigma_1=\{\tau\}$, and $\Sigma=\Sigma_1\sqcup\cdots\sqcup \Sigma_r$. If we consider the action of $O_F\otimes O_L$ on $\omega$, we have further decompositions
\[\begin{split}\omega&=(\omega_1\oplus\omega_1^c)\oplus\omega_2^c\oplus\cdots\oplus\omega_r^c\\
&=\bigoplus_{\sigma\in\Sigma_1}\omega_\sigma\oplus\bigoplus_{\sigma\in\Sigma_2}\omega_\sigma\oplus\cdots\oplus
\bigoplus_{\sigma\in\Sigma_r}\omega_\sigma\\
&=\omega_{\tau}\oplus\bigoplus_{\sigma\in\Sigma,\sigma\neq\tau}\omega_\sigma,\end{split}\]
where $\omega_{\tau}=\omega_{\tau}^1\oplus\omega_{\tau}^2, \omega_{\tau}^1=\omega_1$ is of rank one, and rank$\omega_{\tau}^2=n-1$. For later use, we make the following definition.
\begin{definition} Let $H=\varepsilon_1^+A[\vp^\infty]$, which is a $p$-divisible group of dimension one and height $n$ over $S_K$. For any scheme $S$ over $O_L$, if $x: S\ra S_K$ is an $S$-valued point of $S_K$, we denote $H_x$ the corresponding base change of $H$ to $S$.
\end{definition}
We have similar objects over $\ov{X}$, $S_{K_0}$ and $\ov{Y}$. In the following, we will also use $H$ to denote a $p$-divisible group of dimension one and height $n$ over a suitable base. The precise meaning of the group $H$ will be clear in the context. By abuse of notation we also denote the restriction of $\omega$ over $X$ by $\omega$, which can be defined directly in the same way.

Consider the subgroup $M'=GL_{n-1}\times GL_1\times\prod_{\sigma\neq\tau}GL_n\subset M_L$, then every irreducible representation of $M'$ can be viewed as an irreducible representation of $M_L$ via the natural projection $M_L\twoheadrightarrow M'$. In the following we will concentrate on these representations and the associated automorphic vector bundles. Let \[\mathcal{T}=Isom(\omega,\mathcal{O}_X^{n-1}\oplus \mathcal{O}_X\oplus\bigoplus_{\sigma\in\Sigma,\sigma\neq\tau}\mathcal{O}_X^n),\] which is a $M'$-torsor over $X$. Denote by $\pi$ the projection $\mathcal{T}\ra X$. For $\kappa=(k_{\tau,1},\dots,k_{\tau,n-1},k_\tau$, $(\overrightarrow{k_{\sigma}})_{\sigma\neq\tau})\in\Z^{n-1}_+\times\Z
\times\prod_{\sigma\neq\tau}\Z_+^n$, let
$\kappa'=(-k_{\tau,n-1},\dots,-k_{\tau,1},k_\tau,(\overrightarrow{k_\sigma})_{\sigma\neq\tau})$.
We define a coherent sheaf
\[\omega^\kappa:=\pi_\ast O_{\mathcal{T}}[\kappa'],\]
which as an automorphic vector bundle corresponds to the irreducible representation $V(\kappa)$ of $M'$ of highest weight $\kappa$. Conversely, any automorphic vector bundle corresponding to an irreducible representation $V(\kappa)$ of $M'$ of highest weight $\kappa$ comes in this way. We note that $V(\kappa)=V_{\kappa'}$ by the notation of \cite{AIP} 2.1: 
\[V_{\kappa'}:=\{f: M'\ra \A^1|\, f(gb)=\kappa'(b)f(g),\, \forall (g,b)\in M'\times (B\cap M')\}. \]
In the above we have denoted $\omega^\kappa=\mathcal{V}_{\kappa}$.
\begin{definition}We call the elements of $H^0(X,\omega^\kappa)$ automorphic forms of weight $\kappa$ over $X$.
\end{definition}
By definition, an element $f\in H^0(X,\omega^\kappa)$ is a functorially defined function $\mathcal{T}(R)\ra R$ for any $L$-algebra $R$, such that for all $b\in B_{M'}:=B\cap M'$ we have \[f(A,\lambda,\iota,\ov{\eta}, ((v_\tau^+,v_\tau^-),(v_\sigma)_{\sigma\neq\tau})\circ b)=\kappa'(b)f(A,\lambda,\iota,\ov{\eta}, ((v_\tau^+,v_\tau^-),(v_\sigma)_{\sigma\neq\tau})),\]
where $(A,\lambda,\iota,\ov{\eta}, ((v_\tau^+,v_\tau^-),(v_\sigma)_{\sigma\neq\tau}))$ is a $Spec(R)$- valued point of $\mathcal{T}$.

The sheaf $\omega^\kappa$ has a decomposition as tensor product
\[\omega^\kappa=\omega^\kappa_\tau\otimes\bigotimes_{\sigma\neq\tau}\omega^\kappa_\sigma,\]which corresponds to the factors of $\kappa=(\kappa_\tau,(\kappa_\sigma)_{\sigma\neq\tau})$.
Let $T'=T_\tau\times\prod_{\sigma\in\Sigma,\sigma\neq\tau}T_\sigma$ be the maximal torus of $M'$. Let $(\overrightarrow{k_\sigma})_{\sigma\neq\tau}(=(\kappa_\sigma)_{\sigma\neq\tau})\in\prod_{\sigma\neq\tau}\Z^{n}_+$ be fixed. Under the assumption $F_{\vp}=\Q_p$, $T_\tau$ is defined over $\Q_p$. Let $\W$ be the rigid analytic space over $L$ associated to the Iwasawa algebra $O_L[[T_\tau(\Z_p)]]$, then it is defined over $\Q_p$. Its points have the following explanation: for any affinoid algebra $A$ over $L$, $\W(A)$ is the set of continuous hommorphisms $Hom(T_\tau(\Z_p), A^\times)$. We have an embedding:
\[\begin{split}\W(\Cm_p)=Hom(T_\tau(\Z_p),\Cm_p^\times)&\hookrightarrow Hom(\prod_{\sigma\in\Sigma}T_\sigma,\Cm_p^\times)\\
\kappa_\tau&\longmapsto (\kappa_\tau,(\overrightarrow{k_\sigma})_{\sigma\neq\tau}).\end{split}\]We will identify $\W(\Cm_p)$ with its image under the above embedding, and call it the weight space. In particular, it contains the set of integral weights $\Z^{n-1}_+\times\Z\times(\overrightarrow{k_\sigma})_{\sigma\neq\tau}$. The goal of this paper is to put the automorphic bundles associated to integral weights into a family. More precisely, we will define a sheaf
$\omega^{\kappa_\tau}$ for any $\kappa_\tau\in\W$ over the rigid analytic Shimura variety $X^{rig}$ associated to $X$ over $L$, which after tensor product with $\otimes_{\sigma\neq\tau}\omega^\kappa_\sigma$ interpolates the classical automorphic bundles in some suitable sense. In fact, we can only define these sheaves over some admissible opens of $X^{rig}$. So we need better understanding of the geometry of $X^{rig}$, which is closely related to the geometry of the special fiber $\ov{X}$. Recall that by GAGA, $H^0(X,\omega^\kappa)=H^0(X^{rig},\omega^\kappa)$ for the integral weight $\kappa$.

\subsection{The geometry of special fibers and rigid generic fibers}
Before going further into the theory of automorphic forms, we shall review some basic geometric facts on the special fibers and rigid generic fibers of $S_{K_0}$ and $S_K$.

The Newton stratification of $\ov{Y}$ was studied in detail in the section III.4 of \cite{HT}. Recall this stratification
\[\ov{Y}=\coprod_{h=0}^{n-1}\ov{Y}^h,\]
where a point $x\in\ov{Y}^h$ if and only if $htH_x^{et}=h$. Here $H_x^{et}$ is the \'etale part of one dimensional $p$-divisible group $H_x$ associated to $x$. Here and in the following, for a $p$-divisible group $H$ over some suitable base, we denote by $htH$ the height of $H$. We have $dim\ov{Y}^h=h$. In particular $\ov{Y}^{n-1}$ is the $\mu$-ordinary locus, which is open and dense in $\ov{Y}$. The usual ordinary locus of $\ov{Y}$ is non empty if and only if $F_\vp=\Q_p$ (\cite{W}, 1.6.3), in which case it is also open and dense. 
\begin{assumption}
In the following of this paper, we will assume $F_\vp=\Q_p$ to use the theory of canonical subgroups.
\end{assumption}
See remark 6.5 in the final section for some general cases. Consider the rigid analytic space $Y^{rig}$ associated to $Y$. 
\begin{definition}
 We have the stratification of rigid analytic spaces over $L$
\[Y^{rig}=\coprod_{h=0}^{n-1}]\ov{Y}^h[,\]
here $]\ov{Y}^h[\subset Y^{rig}$ is the tube over $\ov{Y}^h$. 
\end{definition}
For a point $x\in ]\ov{Y}^h[$, we have the associated one dimensional $p$-divisible group $H_x/O_{k(x)}$ with $htH_x=n$. Here and in the following $k(x)$ is the residue field of $x$ and $O_{k(x)}$ is the integer ring. We have the local-\'etale exact sequence of $p$-divisible groups over $O_{k(x)}$:
\[0\ra H_x^0\ra H_x\ra H_x^{et}\ra 0,\]
with $htH_x^0=n-h, htH_x^{et}=h$.

Recall that $S_{K}$ parameterizes the set of total flags of $H[p]$, where $H/S_{K_0}$ is the universal one dimensional $p$-divisible group over $S_{K_0}$, see \cite{TY} section 3. As in the last subsection, let $X^{rig}$ be the rigid analytic space associated to $X$ over $L$. Now for the varieties $\ov{X}, X^{rig}$, we have also the Newton stratification
\[\ov{X}=\coprod_{h=0}^{n-1}\ov{X}^h, X^{rig}=\coprod_{h=0}^{n-1}]\ov{X}^h[.\]
In our special unitary case of signature $(1,n-1)\times(0,n)\times\cdots\times(0,n)$, the Newton stratification and the $p$-rank stratification coincide.
As in \cite{TY} Proposition 3.4 (3), we know that the irreducible components of $\ov{X}$ are
\[X_i=\{x\in\ov{X}|\,\textrm{Fil}_{ix} \,\textrm{is connected}\,\}\]for the filtration $0\subsetneq \textrm{Fil}_{1x}\subsetneq\cdots\subsetneq \textrm{Fil}_{(n-1)x}\subsetneq H_x[p]$. For any non empty subset $S\subset\{1,\dots,n\}$, let $X_S=\bigcap_{i\in S}X_i, X_S^0=X_S-\bigcup_{i\in S,j\in \{1,\dots,n\}-S}(X_i\cap X_j)$, then by loc. cit.
\[\ov{X}^h=\coprod_{\sharp S=n-h}X_S^0.\]
In particular, $]\ov{X}^{n-1}[=\coprod_{i=1}^n]X_i^0[$. The locus
\[]X_1^0[\subset ]\ov{X}^{n-1}[\]is sometimes called ``multiplicative-ordinary'' locus in our case $F_\vp=\Q_p$.

\section{Overconvergent automorphic forms}

\subsection{Canonical subgroups and applications}

Let the notation and assumption be as before. Let $H/\ov{Y}$ be the universal $p$-divisible group of dimension one over the special fiber $\ov{Y}$. Then we have the Frobenius and Verschiebung morphisms $F: H\ra H^{(p)}, V: H^{(p)}\ra H$. Let \[V^\ast: \omega_{H}\ra \omega_{H}^{\otimes p}\]be the induced morphism on cotangent bundles. Then the determinant of $V^\ast$ defines a section $Ha\in H^0(\ov{Y},\omega_{H}^{\otimes(p-1)})$. We know that
the non-vanishing locus of $Ha$ is the ordinary locus $\ov{Y}^{n-1}$.

Let $K$ be a complete valued extension of $\Q_p$ for a valuation $v: K\ra\R\cup\{\infty\}$ such that $v(p)=1$ (Here and in the following, since the level is fixed, by abuse of notation $K$ will denote a complete extension of $\Q_p$. In any case, the precise meaning should be clear from the context.). We denote by $O_K$ the ring of integers of $K$ and let $v: O_K/pO_K\ra [0,1]$ be the truncated valuation. For any $w\in v(O_K)$ we set $\mathfrak{m}(w)=\{x\in K|\,v(x)\geq w\}$ and $O_{K,w}=O_K/\mathfrak{m}(w)$. Let $H/O_K$ be a $p$-divisible of dimension one, and $G=H[p]$ be the truncated $p$-divisible group of level one. Consider the group $G_{O_K/pO_K}$. Then we have an element $Ha(H) \in \omega_{G}^{\otimes(p-1)}$ in the $O_{K,1}$-module
$\omega_{G}^{\otimes(p-1)}$. Let $Hdg(H):=v(Ha(H))\in [0,1]$. Recall that we have the following theorem.
\begin{theorem}[\cite{F2}, Th\'eor\`eme 6]
Let $m\geq1$ be an integer. Assume that $Hdg(H)<\frac{1}{2p^{m-1}} (resp. \, \frac{1}{3p^{m-1}}\, if p=3)$. Then the first step of the Harder-Narasimhan filtration of $H[p^m]$, denoted by $C_m$, is called the canonical subgroup of level $m$ of $H$. It has the following properties.
\begin{enumerate}
\item $C_m(\ov{K})\simeq \Z/p^m\Z$.
\item $deg C_m=1-\frac{p^m-1}{p-1}Hdg(H)$.
\item For any $1\leq k< m$, $C_m[p^k]$ is the canonical subgroup of level $k$ of $H$.
\item In $H_{O_{K,1-Hdg(H)}}$ we have that $C_{1O_{K,1-Hdg(H)}}$ is the kernel of Frobenius.
\item For any $1\leq k<m, Hdg(H/C_k)=p^kHdg(H)$ and $C_m/C_k$ is the canonical subgroup of level $m-k$ of $H/C_k$.
\item Let $H^D$ be the Cartier-Serre dual of $H$. Denote by $C_m^\bot$ the annihilator of $C_m$ under the natural pairing $H[p^m]\times H^D[p^m]\ra \mu_{p^m}$. Then $Hdg(H^D)=Hdg(H)$ and $C_m^\bot$ is the canonical subgroup of level $m$ of $H^D$.
\end{enumerate}
\end{theorem}

Let $\mathcal{Y}/SpfO_L$ be the formal completion of $S_{K_0}$ along the special fiber $\ov{Y}$. Since the variety $Y$ is proper, we have $\mathcal{Y}^{rig}=Y^{rig}$.  Then by \cite{F2} 2.2.2 there is a continuous function
\[Hdg: Y^{rig}\ra [0,1]\]such that for any point $x\in Y^{rig}, Hdg(x)=Hdg(H_x)$. Here $H_x$ is the $p$-divisible group over $O_{k(x)}$ associated to $x$. By our construction, we have
\[]\ov{Y}^{n-1}[=Hdg^{-1}(0),\]and for any $\varepsilon\in [0,1]\cap v(\ov{L}^\times)$,
\[Y(\varepsilon):=Hdg^{-1}([0,\varepsilon])\] is an admissible open subset of $Y^{rig}$, which is a strict neighborhood of $]\ov{Y}^{n-1}[$.

As in \cite{AIP} 4.1, we let \textbf{Adm} be the category of admissible $O_L$-algebras, i.e. flat $O_L$-algebras which are quotients of rings of restricted power series $O_L\langle X_1,\dots,X_r\rangle$ for some $r>0$. Let \textbf{NAdm} be the category of normal admissible $O_L$-algebras. For any object $R$ of \textbf{Adm}, we let $R$-\textbf{Adm} be the category of $R$-algebras which are admissible as $O_L$-algebras. We define similarly $R$-\textbf{NAdm}.

Fix an object $R$ of \textbf{NAdm}. Let $S=Spec R$ and $S^{rig}$ be the rigid analytic space associated to the formal scheme $\wh{S}:=Spf R$. Let $H/S$ be a $p$-divisible group of dimension one and constant height $n$. Assume that there is a $v<\frac{1}{2p^{m-1}}$ (resp. $v<\frac{1}{3p^{m-1}}$ if $p=3$) such that for any $x\in S^{rig}, Hdg(x)<v$. Then for any point $x\in S^{rig}$, $H_x$ has a canonical subgroup of level $m$. By the properties of the Harder-Narasimhan filtration, there is a finite flat subgroup $C_{m,L}\subset H_{S^{rig}}$ interpolating the canonical subgroups of level $m$ for all the points $x\in S^{rig}$. If $w\in v(O_L)$, we set $R_w=R\otimes O_{L,w}$ and for any $R$-module $M$, set $M_w=M\otimes R_w$.
\begin{proposition}[\cite{AIP}, Prop.4.1.3, 4.2.1, 4.2.2]
\begin{enumerate}
\item The subgroup $C_{m,L}$ extends to a finite flat subgroup scheme $C_m\subset H[p^m]$ over $S$.
\item Let $w\in v(O_L)$ with $w\leq m-v\frac{p^m-1}{p-1}$. The morphism of coherent sheaves $\omega_{H[p^m]}\ra \omega_{C_m}$ induces an isomorphism $\omega_{H[p^m],w}\ra\omega_{C_m,w}$.
\item Assume $C_m^D(R)\simeq \Z/p^m\Z$. Then the cokernel of the linearized Hodge-Tate map
\[HT_{C_m^D}\otimes1: C_m^D(R)\otimes R\ra \omega_{C_m}\] is killed by $p^{\frac{v}{p-1}}$.
\end{enumerate}
\end{proposition}
Similarly, we have a finite flat subgroup scheme $C_m^\bot\subset H^D[p^m]$ over $S$ with the same properties as in the above (2) and (3), whose generic fiber interpolates the canonical subgroups of level $m$ in $H^D_x$ for all points $x\in S^{rig}$.

We fix a rational number $v$ such that $v<\frac{1}{2p^{m-1}}$ (resp. $v<\frac{1}{3p^{m-1}}$ if $p=3$) with the property that for any point $x\in S^{rig}, Hdg(x)<v$. Let $C_m$ denote the canonical subgroup of $H$ of level $m$ over $S$. Consider the dual, we have the canonical subgroup $C^\bot_m$ of $H^D$ of level $m$ over $S$. In the following, we will assume that $C_m^D(R)\simeq \Z/p^m\Z, (C^\bot_m)^D(R)\simeq (\Z/p^m\Z)^{n-1}$, and the sheaves $\omega_H$ and $\omega_{H^D}$ are free over $S$. Then Prop. 4.3.1 of \cite{AIP} gives us the following sheaves.
\begin{itemize}
\item
There is a free sub-sheaf of $R$-modules $\Fm^+$ of $\omega_H$ of rank 1 containing $p^{\frac{v}{p-1}}\omega_H$ which is equipped, for all $w\in ]0,m-v\frac{p^m}{p-1}]$, with a map
\[HT_w^+: C_m^D(R[1/p])\ra \Fm^+\otimes_R R_w\]
deduced from $HT_{C_m^D}$ which induces an isomorphism
\[HT_w^+\otimes 1: C_m^D(R[1/p])\otimes_\Z R_w\ra \Fm^+\otimes_R R_w.\]
\item There is a free sub-sheaf of $R$-modules $\Fm^-$ of $\omega_{H^D}$ of rank $n-1$ containing $p^{\frac{v}{p-1}}\omega_{H^D}$ which is equipped, for all $w\in ]0,m-v\frac{p^m}{p-1}]$, with a map
\[HT_w^-: (C_m^\bot)^D(R[1/p])\ra \Fm^-\otimes_R R_w\]
deduced from $HT_{(C_m^\bot)^D}$ which induces an isomorphism
\[HT_w^-\otimes 1: (C_m^\bot)^D(R[1/p])\otimes_\Z R_w\ra \Fm^-\otimes_R R_w.\]
\end{itemize}
Then we have the sum maps
\[HT_w=HT_w^+\oplus HT_w^-:  C_m^D(R[1/p])\oplus(C_m^\bot)^D(R[1/p])\ra \Fm^+\otimes_R R_w\oplus\Fm^-\otimes_R R_w,\]
\[HT_w\otimes1: C_m^D(R[1/p])\otimes_\Z R_w\oplus(C_m^\bot)^D(R[1/p])\otimes_\Z R_w\ra \Fm^+\otimes_R R_w\oplus\Fm^-\otimes_R R_w.\]

We keep the above notations and assumptions. Let $\mathcal{GR}\ra S$ be the flag variety parameterizing all (total) flags \[\textrm{Fil}_0\Fm^-=0\subset \textrm{Fil}_1\Fm^-\subset\dots\subset \textrm{Fil}_{n-1}\Fm^-=\Fm^-\] of the free module $\Fm^-$. We can view it parameterizes flags of \[\Fm:=\Fm^-\oplus\Fm^+\] of the form $\textrm{Fil}_i(\Fm)=\textrm{Fil}_i\Fm^-$ for $i=0,1,\dots,n-1$ and $\textrm{Fil}_n\Fm=\Fm$. Note that when $n=2$, we have $\mathcal{GR}=S$. Let $\mathcal{GR}^+$ be the $T_\tau=(\G_m)^n$-torsor over $\mathcal{GR}$ which parameterizes flags $\textrm{Fil}_\bullet\Fm$ together with basis $v_i$ of the graded pieces $Gr_i\Fm^-$ for $1\leq i\leq n-1$ and basis $v_n$ of $\Fm^+$.

We fix isomorphisms \[\psi^+: \Z/p^m\Z\simeq C_m^D(R[\frac{1}{p}]), \quad \psi^-: (\Z/p^m\Z)^{n-1}\simeq (C_m^\bot)^D(R[\frac{1}{p}]).\] Let $x_1,\dots,x_{n-1},x_n$ be the $\Z/p^m\Z$-basis of $(C_m^\bot)^D(R[\frac{1}{p}])\oplus C_m^D(R[\frac{1}{p}])$ corresponding to the canonical basis of $(\Z/p^m\Z)^{n-1}\oplus\Z/p^m\Z$. By $\psi=(\psi^+, \psi^-)$ we obtain a flag \[\textrm{Fil}_\bullet^{\psi}=\{0\subset \lan x_1\ran\subset\lan x_1,x_2\ran\subset\dots\subset\lan x_1,\dots,x_n\ran=(C_m^\bot)^D(R[\frac{1}{p}])\oplus C_m^D(R[\frac{1}{p}])\}.\] Let $\ov{x_i}$ be the basis of the graded pieces. Let $R'$ be an object in $R$-\textbf{Adm}. We say that an element $\textrm{Fil}_\bullet\Fm\otimes R'\in \mathcal{GR}(R')$ is $w$-compatible if \[\textrm{Fil}_\bullet\Fm\otimes R'_w=HT_w(\textrm{Fil}_\bullet^{\psi})\otimes R_w'.\] We say that an element $(\tr{Fil}_\bullet\Fm\otimes R',\{v_i\})\in\mathcal{GR}^+(R')$ is $w$-compatible if \[\tr{Fil}_\bullet\Fm\otimes R'_w=HT_w(\tr{Fil}_\bullet^{\psi})\otimes R_w'\] and \[v_i \,\mathrm{mod} p^w\Fm\otimes R'+ \tr{Fil}_{i-1}\Fm\otimes R'=HT_w(\ov{x_i}).\]

We define functors
\[\begin{split}\mathfrak{IW}_w: R-\mathbf{Adm}& \ra SET\\
R'&\mapsto \{w-\tr{compatible}\;\tr{Fil}_\bullet\Fm\otimes R'\in \mathcal{GR}(R')\},\\
\mathfrak{IW}_w^+: R-\mathbf{Adm}& \ra SET\\
R'&\mapsto \{w-\tr{compatible}\,(\tr{Fil}_\bullet\Fm\otimes R',\{v_i\})\in \mathcal{GR}^+(R')\}.\end{split}\]
These two functors are representable by affine formal schemes, for more detailed description see \cite{AIP} 4.5. We only remark that $\mathfrak{IW}_w^+$ is a torsor over $\mathfrak{IW}_w$ under $\mathfrak{T}_w$. Where $\mathfrak{T}_w$ is the formal torus defined by \[\mathfrak{T}_w(R')=Ker(T_\tau(R')\ra T_\tau(R'/\varpi^wR'))\] for any object $R'$ in \textbf{Adm}. All these constructions are functorial in $R$. They do not depend on $m$ but only on $w$.

\subsection{The overconvergent sheaves $\omega_w^{\dag\kappa}$}
Recall we have the Shimura variety $Y$. On the associated rigid space $Y^{rig}$, we have a continuous function $Hdg: Y^{rig}\ra [0,1]$. For $v\in[0,1]$, we have the open subset $Y(v)=Hdg^{-1}([0,v])$. There is a formal model $\mathfrak{Y}(v)$ of $Y(v)$ by some suitable blow-up and normalization, see \cite{AIP} 5.2.

Let $m\geq 1$ be an integer and $v\in v(O_K)$ such that $v<\frac{1}{2p^{m-1}}$ (resp. $v<\frac{1}{3p^{m-1}}$ if $p=3$). We have canonical subgroups $C_m, C_m^\bot$ of level $m$ over $Y(v)$. Let
\[X_1(p^m)(v)=Isom_{Y(v)}(\Z/p^m\Z,C_m^D)\times Isom_{Y(v)}((\Z/p^m\Z)^{n-1},(C_m^\bot)^D).\] It is a finite \'etale cover of $Y(v)$. Let $\mathfrak{X}_1(p^m)(v)$ be the normalization of $\mathfrak{Y}(v)$ in $X_1(p^m)(v)$. Let $B_\tau\subset GL_{n-1}\times GL_1$ be the Borel subgroup which contains $T_\tau$ with unipotent radical $U_\tau$ (when $n=2$, $B_\tau=T_\tau, U_\tau$ is trivial). Set $\mathfrak{X}(p^m)(v)= \mathfrak{X}_1(p^m)(v)/B_\tau(\Z/p^m\Z)$. We have the following modular interpretations for the formal schemes $\mathfrak{X}_1(p^m)(v)$ and $\mathfrak{X}(p^m)(v)$.
\begin{proposition}
For any object $R$ in \textbf{NAdm},
\begin{enumerate}
\item $\mathfrak{X}_1(p^m)(v)(R)$ is the set of isomorphic classes of $(A,\iota,\lambda,\ov{\eta},\psi^+,\psi^-)$, where $(A,\iota,\lambda,\ov{\eta})\in \mathcal{Y}(R)$, and for any rigid point $x$ in $R$, $Hdg(H_x)\leq v$; $\psi^+: \Z/p^m\Z\simeq C_m^D, \psi^-: (\Z/p^m\Z)^{n-1}\simeq (C_m^\bot)^D$ are trivialization of canonical subgroups of level $m$ over $R[\frac{1}{p}]$.
\item $\mathfrak{X}(p^m)(v)(R)$ is the set of isomorphic classes $(A,\iota,\lambda,\ov{\eta},\tr{Fil}_\bullet)$, where $(A,\iota,\lambda,\ov{\eta})\in \mathcal{Y}(R)$, and for any rigid point $x$ in $R$, $Hdg(H_x)\leq v$; $\tr{Fil}_\bullet$ is a full flag $\tr{Fil}_\bullet$ of $H[p^m]$ over $R[\frac{1}{p}]$ such that $\tr{Fil}_{1}=C_m$.
\end{enumerate}
\end{proposition}
\begin{proof}
(1) is clear from the construction. For (2), by definition $\mathfrak{X}(p^m)(v)(R)$ is the set of isomorphic classes $(A,\iota,\lambda,\ov{\eta},\tr{Fil}_\bullet^+,\tr{Fil}_\bullet^-)$, where $(A,\iota,\lambda,\ov{\eta})\in \mathcal{Y}(R)$, and for any rigid point $x$ in $R$, $Hdg(H_x)\leq v$; $\tr{Fil}^+_\bullet$ (resp. $\tr{Fil}^-_\bullet$) is a full flag of $C_m$ (resp. $C_m^\bot$). One can easily translate $\tr{Fil}^+_\bullet$ and $\tr{Fil}^-_\bullet$ as a full flag of $H[p^m]$ such that $\tr{Fil}_{1}=C_m$.
\end{proof}
We will identify the formal scheme $\mathfrak{X}(p)(v)$ as a sub formal scheme of $\mathfrak{X}$, and simply write it as $\mathfrak{X}(v)$.

Let $w\in v(O_L)\cap ]m-1+\frac{v}{p-1},m-v\frac{p^m}{p-1}]$. Let $H/\mathfrak{X}_1(p^m)(v)$ be the universal $p$-divisible group. Applying the construction in the last subsection, we have locally free sub-sheaves $\Fm^+\subset \omega_{H,\mathfrak{X}_1(p^m)(v)}, \Fm^-\subset\omega_{H^D,\mathfrak{X}_1(p^m)(v)}$. They are equipped with isomorphisms:
\[(HT_w\circ\psi^+)\otimes1: \Z/p^m\Z\otimes \mathcal{O}_{\mathfrak{X}_1(p^m)(v)}/p^w\mathcal{O}_{\mathfrak{X}_1(p^m)(v)}\simeq \Fm^+\otimes O_{L,w},\]
\[(HT_w\circ\psi^-)\otimes1: (\Z/p^m\Z)^{n-1}\otimes \mathcal{O}_{\mathfrak{X}_1(p^m)(v)}/p^w\mathcal{O}_{\mathfrak{X}_1(p^m)(v)}\simeq \Fm^-\otimes O_{L,w}.\]

We have a chain of formal schemes:
\[\mathfrak{IW}_w^+\st{\pi_1}{\ra}\mathfrak{IW}_w\st{\pi_2}{\ra}\mathfrak{X}_1(p^m)(v)
\st{\pi_3}{\ra}\mathfrak{X}(p^m)(v)\st{\pi_4}{\ra}\mathfrak{X}(v).\]
Recall that $\mathfrak{IW}_w^+$ is a torsor over $\mathfrak{IW}_w$ under the formal torus $\mathfrak{T}_w$.
Let $\mathfrak{B}_w$ be the formal group defined by
\[\mathfrak{B}_w(R)=Ker(B_\tau(R)\ra B_\tau(R/p^wR))\]for all $R\in$ \textbf{Adm}. Then there is a surjective map $\mathfrak{B}_w\ra\mathfrak{T}_w$ with kernel $\mathfrak{U}_w$. Then we have an action of $B_\tau(\Z_p)\mathfrak{B}_w$ on $\mathfrak{IW}_w^+$ over $\mathfrak{X}(p^m)(v)$ (with $\mathfrak{U}_w$ acting trivially).

Recall our weight space $\W$ with $\W(\Cm_p)=Hom(T_\tau(\Z_p),\Cm_p^\times)$. As in Definition 2.2.1 of \cite{AIP}, for $w\in\Q_{>0}$, a character $\kappa\in\W(\Cm_p)$ is called $w$-analytic if it extends to an analytic map $\kappa: T_\tau(\Z_p)(1+p^wO_{\Cm_p})^n\ra \Cm_p^\times$. Moreover, $\W$ has an increasing cover by affinoids $\W=\bigcup_{w>0}\W(w)$, such that the restriction of the universal character $\kappa^{un}$ of $\W$ to $\W(w)$ is $w$-analytic. Let $\kappa=\kappa_\tau\in\W(K)$ be a $w$-analytic character, where $K|L$ is a finite extension. Then we have the total character $\widetilde{\kappa}=(\kappa,(\kappa_\sigma)_{\sigma\neq\tau})$, with the characters $(\kappa_\sigma)_{\sigma\neq\tau}$ fixed as in the subsection 2.2. The involution \[\kappa=(k_{\tau,1},\dots,k_{\tau,n-1},k_\tau)\mapsto\kappa'=(-k_{\tau,n-1},\dots,-k_{\tau,1},k_\tau)\] of $X^\ast(T_\tau)$ extends to an involution of $\W$ mapping $w$-analytic characters to $w$-analytic characters. The character $\kappa': T_\tau(\Z_p)\ra K^\times$ extends to a character $\kappa': B_\tau(\Z_p)\mathfrak{B}_w\ra K^\times$ with $U_\tau(\Z_p)\mathfrak{U}_w$ acting trivially. Recall that we have the above chain of morphisms of formal schemes.
Set $\pi=\pi_1\circ\pi_2\circ\pi_3\circ\pi_4$. Let $\wt{\omega^\tau}$ be the formal completion of the integral bundle $\otimes_{\sigma\neq\tau}\omega^{\kappa_\sigma}_\sigma$ over $S_K$. This is a sheaf over $\mathfrak{X}$. Since $(\kappa_\sigma)_{\sigma\neq\tau}$ is fixed, the weights of our automorphic forms depend only on $\kappa=\kappa_\tau$.
\begin{definition}
\begin{enumerate}
\item The formal Banach sheaf of $w$-analytic, $v$-overconvergent automorphic forms of weight $\kappa$ is
\[\mathfrak{m}_w^{\dag\kappa}=\pi_\ast \mathcal{O}_{\mathfrak{IW}_w^+}[\kappa']\otimes\wt{\omega^\tau}|_{\mathfrak{X}(v)}.\]
\item The space of integral $w$-analytic, $v$-overconvergent automorphic forms of weight $\kappa$ over $X$ is
\[M_w^{\dag\kappa}(\mathfrak{X}(v))=H^0(\mathfrak{X}(v), \mathfrak{m}_w^{\dag\kappa}).\]
\end{enumerate}
\end{definition}

Let $\kappa,m,v,w$ satisfy all the compatible conditions for the existence of $\mathfrak{m}^{\dag\kappa}_w$. If $v'<v$ then $\kappa,m,v',w$ satisfy also the conditions and the sheaf $\mathfrak{m}^{\dag\kappa}_w$ on $\mathfrak{X}(v')$ is the restriction of the sheaf on $\mathfrak{X}(v)$. For any $w'>w$, one can find $m'$ such that $\kappa,m',v,w'$ satisfy the conditions, and one has a map $\mathfrak{m}^{\dag\kappa}_w\ra\mathfrak{m}^{\dag\kappa}_{w'}$ and thus a map $M_w^{\dag\kappa}(\mathfrak{X}(v))
\ra M_{w'}^{\dag\kappa}(\mathfrak{X}(v))$.
\begin{definition}
Let $\kappa\in \W$. The space of integral locally analytic overconvergent automorphic forms of weight $\kappa$ over $X$ is
\[M^{\dag\kappa}(\mathfrak{X})=\varinjlim_{v\ra0,w\ra\infty}M_w^{\dag\kappa}(\mathfrak{X}(v)).\]
\end{definition}

Let $\mathcal{IW}_w^+$ and $\mathcal{IW}_w$ be the rigid spaces associated to $\mathfrak{IW}^+_w$ and $\mathfrak{IW}_w$ respectively. They are admissible opens of the rigid spaces associated to $\mathcal{GR}^+$ and $\mathcal{GR}$ respectively. Let $T_w$ be the rigid space associated to the formal torus $\mathfrak{T}_w$. Then $\mathcal{IW}_w^+$ is a $T_w$-torsor over $\mathcal{IW}_w$. The rigid spaces associated to $\mathfrak{X}(p^m)(v)$ (resp. $\mathfrak{X}(v)$) will be denoted by $X(p^m)(v)$ (resp. $X(v)$). Note that $X(v)$ is a strict neighborhood of the tube $]X_1^0[$ over the multiplicative-ordinary locus, see subsection 2.3. Moreover, $X(0)=]X_1^0[$. We have a chain of rigid spaces:
\[\mathcal{IW}_w^+\ra \IW_w\ra X_1(p^m)(v)\ra X(p^m)(v)\ra X(v).\]
As in \cite{AIP}, let $X^+(p^m)(v)=X_1(p^m)(v)/U_\tau(\Z/p^m\Z)$, which is the rigid space associated to $\mathfrak{X}^+(p^m)(v)= \mathfrak{X}_1(p^m)(v)/U_\tau(\Z/p^m\Z)$, then $\IW_w^+$ (resp. $\IW_w$) descends to a rigid space $\IW_w^{0+}$ (resp. $\IW_w^0$) over $X^+(p^m)(v)$ (resp. $X(p^m)(v)$). Moreover, $\IW_w^{0+}$ is a $T_w$ torsor over $\IW_w^0$. Recall the $M'=GL_{n-1}\times GL_1\times\prod_{\sigma\neq\tau}GL_n$-torsor $\mathcal{T}$ over $X$. We have the decomposition $\mathcal{T}=\mathcal{T}_\tau\times\mathcal{T}^\tau$, where $\mathcal{T}_\tau$ (resp. $\mathcal{T}^\tau$) is the $GL_{n-1}\times GL_1$-torsor (resp. $\prod_{\sigma\neq\tau}GL_n$-torsor) over $X$. Let $\T^{rig}, \T_\tau^{rig}, \T^{\tau,rig}$ be the rigid analytic spaces associated to $\T, \T_\tau$ and $\T^\tau$ respectively. For $w>m-1+\frac{v}{{p-1}}$, we have open immersions (\cite{AIP} Prop. 5.3.1) \[\IW_w^{0+}\hookrightarrow (\T^{rig}_\tau/U_{\tau})|_{X(v)}, \IW_w^{0}\hookrightarrow (\T^{rig}_\tau/B_{\tau})|_{X(v)}.\]

Let $\omega_w^{\dag\kappa}$ be the generic fiber of the formal Banach sheaf $\mathfrak{m}_w^{\dag\kappa}$. It can be defined by using the morphism $\IW_w^{0+}\ra X(v)$ in the same way as in the definition of $\mathfrak{m}_w^{\dag\kappa}$.
\begin{definition}
Let $\kappa\in\W$. The space of $w$-analytic, $v$-overconvergent automorphic forms of weight $\kappa$ is
\[M_w^{\dag\kappa}(X(v))=H^0(X(v),\omega_w^{\dag\kappa}).\]
The space of locally analytic overconvergent automorphic forms of weight $\kappa$ is
\[M^{\dag\kappa}(X)=\varinjlim_{v\ra0,w\ra\infty}M_w^{\dag\kappa}(X(v)).\]
\end{definition}

Concretely, we can describe a $w$-analytic, $v$-overconvergent automorphic form $f$ of weight $\kappa$ as an element in $H^0(\IW_w^{0+}\times\T^{\tau,rig},\mathcal{O}_{\IW_w^{0+}\times\T^{\tau,rig}})$ in the following way (here and in the following we write simply $\IW_w^{0+}\times\T^{\tau,rig}$ for the space $\IW_w^{0+}\times\T^{\tau,rig}|_{X(v)}$ over $X(v)$). For any finite extension $K$ of $L$, a $K$-valued point of $\IW_w^{0+}\times\T^{\tau,rig}$ has the form \[(A,\lambda,\iota,\ov{\eta}, \tr{Fil}_\bullet H[p], \tr{Fil}_\bullet\Fm, (v_\tau,(v_\sigma)_{\sigma\neq\tau})),\]where $(A,\lambda,\iota,\ov{\eta})\in Y(v)(K)$, $\tr{Fil}_\bullet H[p]$ is a full flag of the $p$-torsion subgroup of the associated one dimensional $p$-divisible group $H$ over $O_K$ such that $\tr{Fil}_1H[p]=C_1$, $(\tr{Fil}_\bullet\Fm,v_\tau)$ is a full flag in $(\T_\tau^{rig}/B_\tau)(K)$ with a basis $v_\tau=(v_{1\tau},\dots,v_{(n-1)\tau},v_{n\tau})$ for the graded pieces, such that there is a trivialization $\psi=(\psi^+,\psi^-): C_m^D(\ov{K})\oplus (C_m^\bot)^D(\ov{K})\simeq \Z/p^m\Z\oplus(\Z/p^m\Z)^{n-1}$ (for some integer $m\geq 1$) which is compatible with $\tr{Fil}_\bullet H[p]$, and $(\tr{Fil}_\bullet\Fm,v_\tau)$ is $w$-compatible with $\psi$; finally $(v_\sigma)_{\sigma\neq\tau}$ is an element in $\T^{\tau,rig}(K)$ over $X^{rig}(K)$. Then \[f(A,\lambda,\iota,\ov{\eta}, \tr{Fil}_\bullet H[p], \tr{Fil}_\bullet\Fm, (v_\tau,(v_\sigma)_{\sigma\neq\tau}))
\in K\] such that for all $b\in B_{M'}$ we have
\[f(A,\lambda,\iota,\ov{\eta}, \tr{Fil}_\bullet H[p], \tr{Fil}_\bullet\Fm, (v_\tau,(v_\sigma)_{\sigma\neq\tau})\circ b)=\wt{\kappa'}(b)f(A,\lambda,\iota,\ov{\eta}, \tr{Fil}_\bullet H[p], \tr{Fil}_\bullet\Fm, (v_\tau,(v_\sigma)_{\sigma\neq\tau})).\]

The space $M_w^{\dag\kappa}(X(v))$ is a Banach space, with the unit ball $M_w^{\dag\kappa}(\mathfrak{X}(v))$. Locally for the \'etale topology, the sheaf $\omega^{\dag\kappa}_w$ has fibres isomorphic to the space $V_{\kappa'}^{w-an}\otimes\bigotimes_{\sigma\neq\tau}V_{\kappa_\sigma}$ (cf. \cite{AIP} Proposition 5.3.4), where $V_{\kappa'}^{w-an}$ is the locally $w$-analytic representation of the Iwahori subgroup $I$ of $GL_{n-1}(\Z_p)\times GL_1(\Z_p)$, which is defined as
\[V_{\kappa'}^{w-an}=\{f: I\ra L|\,f(ib)=\kappa'(b)f(i),\,\forall\, (i,b)\in\,I\times B_\tau(\Z_p),\,f|_{N^0}\in\Fm^{w-an}(N^0,L)\}.\] Here are some explanations. Recall that $I$ has the Iwahori decomposition $I=B_\tau(\Z_p)\times N^0$, where $N^0$ is defined in the following way: let $B_\tau^0\subset GL_{n-1}\times GL_1$ be the opposite Borel subgroup of $B_\tau$ with unipotent radical $U^0_\tau$, then $N^0$ is the subgroup of $U^0_\tau(\Z_p)$ of matrices which reduce to the identity modulo $p$. $\Fm^{w-an}(N^0,L)$ is the set of $w$-analytic functions, i.e. the functions from $N^0$ to $L$ which are the restrictions to $N^0$ of the unique analytic functions on $N^0_w:=\bigcup_{x\in N^0}B(x,p^{-w})$, where $B(x,p^{-w})$ is the closed ball with center $x$ radius $p^{-w}$ in the rigid analytic affine space $\A^{\frac{(n-1)(n-2)}{2},rig}$ over $\Q_p$. Here we identify $N^0$ with $(p\Z_p)^{\frac{(n-1)(n-2)}{2}}\subset \A^{\frac{(n-1)(n-2)}{2},rig}$. In the case $n=2$, we have $N^0=\{1\}, I=B_\tau(\Z_p)=T_\tau(\Z_p)=(\Z_p^\times)^2$, $V_{\kappa'}^{w-an}$ is just the space of $w$-analytic characters.

Recall that we have the total character $\wt{\kappa}=(\kappa,(\kappa_\sigma)_{\sigma\neq\tau})$. When $\kappa\in\Z^{n-1}_+\times\Z$ we have the automorphic vector bundle $\omega^{\wt{\kappa}}$ over $X$, which we will denote simply as $\omega^\kappa$ since $(\kappa_\sigma)_{\sigma\neq\tau}$ is fixed. We denote also the corresponding vector bundle over the rigid analytic space $X^{rig}$ by $\omega^\kappa$. By construction we have the following proposition and corollary.
\begin{proposition}
If $\kappa\in \Z^{n-1}_+\times\Z$, then there is a canonical restriction map
\[\omega^\kappa|_{X(v)}\hookrightarrow \omega_w^{\dag\kappa}\]induced by the open immersion $\IW_w^{0+}\hookrightarrow (\T^{rig}_\tau/U_{\tau})|_{X(v)}$.
Locally for the \'etale topology, this map is isomorphic to the inclusion
\[V_{\kappa'}\otimes V^\tau\hookrightarrow V_{\kappa'}^{w-an}\otimes V^\tau\]of the algebraic induction into the analytic induction, where $V^\tau=\bigotimes_{\sigma\neq\tau}V_{\kappa_\sigma}$.
\end{proposition}
\begin{corollary}
For any $\kappa\in \Z^{n-1}_+\times\Z$, we have an inclusion
\[H^0(X,\omega^\kappa)\hookrightarrow M_w^{\dag\kappa}(X(v))\] from the space of classical forms of weight $\kappa$ into the space of $w$-analytic, $v$-overconvergent automorphic forms of weight $\kappa$.
\end{corollary}

\section{Hecke operators}
\subsection{Hecke operators outside $p$}
Consider the set of prime to $p$ Hecke correspondences $K^p\setminus G(\A_f^p)/K^p$. For any $K^pgK^p$, there is an algebraic correspondence $p_1,p_2: C_g\ra X$ over $X$. It is defined as follows: recall $K=K^pK_p$ with $K_p$ as in subsection 2.2. We have $C_g=Sh_{K_pK^p_g}$ where $K^p_g=K^p\cap g^{-1}K^pg$, with $p_1$ the natural projection and $p_2$ as the composition of the natural projection $Sh_{K_pK^p_g}\ra Sh_{K_pg^{-1}K^pg}$ with the isomorphism $Sh_{K_pg^{-1}K^pg}\ra Sh_{K_pK^p}=X$ induced by $g$. We consider the rigid analytification of this correspondence, then we restrict it over $X(v)$, so we have the diagram \[\xymatrix{
&C_g(v)\ar[ld]_{p_1}\ar[rd]^{p_2}&\\
X(v)& &X(v).
 }\] Over $C_g$ there is a universal prime to $p$ isogeny $\pi: A\ra A'$ defined by the morphism $Sh_{K_pg^{-1}K^pg}\ra Sh_{K_pK^p}=X$, which induces a map $\pi^\ast: \omega_{A'}\ra\omega_A$, hence a map $p_2^\ast(\T^{rig}_\tau/U_\tau\times\T^{\tau,rig})\ra p_1^\ast(\T^{rig}_\tau/U_\tau\times\T^{\tau,rig})$ which is an isomorphism. For $w\in ]m-1+\frac{v}{p-1},m-v\frac{p^m}{p-1}]$, the map $\pi^\ast$ induces an isomorphism (see \cite{AIP} Lemma 6.1)
\[\pi^\ast: p_2^\ast(\IW_w^{0+}\times\T^{\tau,rig})\simeq p_1^\ast(\IW_w^{0+}\times\T^{\tau,rig}).\] We define the Hecke operator $T_g$ as the composition:
\[T_g: H^0(X(v),\omega_w^{\dag\kappa})\st{p_2^\ast}{\ra}H^0(C_g(v),p_2^\ast\omega_w^{\dag\kappa})\st{\pi^{\ast}}{\ra}
H^0(C_g(v),p_1^\ast\omega_w^{\dag\kappa})\st{Trp_1}{\ra}H^0(X(v),\omega_w^{\dag\kappa}).\]

Consider the prime to $p$ Hecke algebra $\mathbb{T}^p=\mathcal{H}(G(\A_f^p)//K^p)$, which is the restricted tensor product of all local Hecke algebras $\mathbb{T}_l=\mathcal{H}(G(\Q_l)//K_l)$ for primes $l\neq p$. We know there are only finite primes $l$ such that $K_l$ is not hyperspecial maximal. Let $\mathbb{T}^{K^p}$ be the sub algebra of $\mathbb{T}^p$ which is the restricted tensor product of all $\mathbb{T}_l$ such that $K^p$ is hyperspecial maximal at $l$. Then $\mathbb{T}^{K^p}$ is commutative. The above construction gives us an action of $\mathbb{T}^{K^p}$ on $M^{\dag\kappa}_w(X(v))$.

\subsection{Hecke operators at $p$}
We will define $n-1$ Hecke operators at $p$ as follows. Recall that in section 2 the primes of the CM field $F$ over $p$ are denoted by $\varpi=\varpi_1,\varpi_2,\dots,\varpi_r$, $\varpi^c,\varpi_2^c,\dots,\varpi_r^c$, and we have assumed $F_\vp=\Q_p$. The $p$-divisible group associated to the universal abelian scheme $A$ over $S_K$ has the decomposition
\[ A[p^\infty]=A[\varpi_1^\infty]\oplus\cdots\oplus A[\varpi_r^\infty]\oplus A[\varpi_1^{c,\infty}]\oplus\cdots\oplus A[\varpi_r^{c,\infty}],\]with $A[\varpi_i^\infty]^D\simeq A[\varpi_i^{c,\infty}]$ for all $1\leq i\leq r$. The $p$-divisible group $H=\varepsilon_1^+ A[\varpi^\infty]$ has dimension one, where $\varepsilon_1^+$ is the idempotent introduced in subsection 2.2. For a finite locally free subgroup $L\subset H[p]$, we denote by $\widetilde{L}\subset A[p]$ the subgroup such that under the induced decomposition $\widetilde{L}=\wt{L}_1\oplus\cdots\oplus\wt{L}_r\oplus\wt{L}^c_1\oplus\cdots\oplus\wt{L}^c_r$, we have
\[\wt{L}_1=\Z_p^n\otimes L, \wt{L}_i=A[\varpi_i],\, 2\leq i\leq r, \wt{L}_1^c=\Z_p^n\otimes(H[p]/L)^D, \wt{L}_i^c=0,\, 2\leq i\leq r.\]Here $\wt{L}_1$ and $\wt{L}_1^c$ are the subgroups of $A[\varpi]$ and $A[\varpi^c]$ respectively via the Morita equivalence corresponding to $L$ and $(H[p]/L)^D$. Note $\varepsilon_1^+(A/\wt{L})[\vp^\infty]=H/L$.

Now for $i=1,\dots,n-1$, let $C_i$ be the moduli scheme over $X$ parameterizing finite locally free subgroups $L\subset H[p]$ such that $L\oplus \tr{Fil}_iH[p]=H[p]$. There are two projections $p_1,p_2: C_i\ra X$. The first projection is defined by forgetting $L$. The second projection is defined by mapping $(A,\lambda,\iota,\ov{\eta},\tr{Fil}_\bullet H[p], L)$ to $((A/\wt{L}),\lambda',\iota',\ov{\eta'}, \tr{Fil}_\bullet(H/L)[p])$, where the filtration on $\varepsilon_1^+(A/\wt{L})[\varpi]=(H/L)[p]$ is defined as follows:
\begin{itemize}
\item For $j=1,\dots,i$, $\tr{Fil}_j(H/L)[p]$ is simply the image of $\tr{Fil}_jH[p]$ in $H/L$,
\item For $j=i+1,\dots,n-1$, $\tr{Fil}_j(H/L)[p]$ is the image of $\tr{Fil}_jH[p]+p^{-1}(\tr{Fil}_jH[p]\cap L)$ in $H/L$.
\end{itemize}
We consider the analytifications $p_1,p_2: C_i^{rig}\ra X^{rig}$. For an admissible open subset $V\subset X^{rig}$, we denote the image of $V$ under the correspondence $C_i$ by $U_i(V):=p_2(p_1^{-1}(V))$, which is also an admissible open subset of $X^{rig}$.

\begin{remark}
In the definition of the correspondences $C_i$ for $i=1,\dots,n-1$, we have taken them by parameterizing $L$ with $L\oplus \tr{Fil}_iH[p]=H[p]$, $\wt{L}_j=A[\varpi_j]$ and $\wt{L}_j^c=0$ for $2\leq j\leq r$. In fact, as in \cite{PS}, it is more natural to define the correspondences $C_{i,\varpi_1}$ for $i=1,\dots,n-1$ by parameterizing $L$ with $L\oplus \tr{Fil}_iH[p]=H[p]$, $\wt{L}_j=\wt{L}_j^c=0$ for $2\leq j\leq r$. At the places $\varpi_j$ for $2\leq j\leq r$, we define correspondences $C_{\varpi_j}$ by parameterizing $L$ with $\wt{L}_j=A[\varpi_j], \wt{L}_k=0$ for all $1\leq k\leq r, k\neq j$. Then $C_{\varpi_j}=X$ for $2\leq j\leq r$ and the morphism $p_2$ is by taking the quotient by $\wt{L}=A[\varpi_j]$. For $2\leq j\leq r$, by construction these quotients do not change the part $\otimes_{\sigma\neq\tau}\omega^\kappa_\sigma$. Therefore when passing to operators on the space of automorphic forms (see below), there is no difference between $U_i=U_{i,\varpi_1}\circ\prod_{j=2}^rU_{\varpi_j}$ and $U_{i,\varpi_1}$ for $i=1,\dots,n-1$.
\end{remark}

We first consider the correspondence $C_1$. If $v<\frac{p-2}{2p-2}$, then by the theory of canonical subgroups we know (\cite{F2} Prop. 17) \[U_1(X(v))\subset X(\frac{v}{p}).\]
Let $C_1(v)=C_1^{rig}\times_{p_1,X^{rig}}X(v)$. We have the diagram
\[\xymatrix{
&C_1(v)\ar[ld]_{p_1}\ar[rd]^{p_2}&\\
X(v)& &X(\frac{v}{p}).
 }\]
Let $\pi: A\ra A'$ be the universal isogeny over $C_1(v)$. Then as before it induces an isomorphism
\[\pi^\ast: p_2^\ast(\IW_w^{0+}\times\T^{\tau,rig})|_{X(\frac{v}{p})}\simeq p_1^\ast(\IW_w^{0+}\times\T^{\tau,rig}).\]
We define the Hecke operator $U_1$ as the composition:
\[H^0(X(\frac{v}{p}),\omega_w^{\dag\kappa})\st{p_2^\ast}{\ra}H^0(C_1(v),p_2^\ast\omega_w^{\dag\kappa})
\st{\pi^{\ast}}{\ra}H^0(C_1(v),p_1^\ast\omega_w^{\dag\kappa})\st{p^{1-n}Trp_1}{\longrightarrow}H^0(X(v),\omega_w^{\dag\kappa}).\]
By abuse of notation, we also denote by $U_1$ the endomorphism of $H^0(X(v),\omega_w^{\dag\kappa})$ obtained as the composition of the above operator we just defined with the restriction map \[H^0(X(v),\omega_w^{\dag\kappa})\ra H^0(X(\frac{v}{p}),\omega_w^{\dag\kappa}).\]

Next we consider the correspondences $C_i$ for $2\leq i\leq n-1$. Let $C_i(v)=C_i^{rig}\times_{p_1,X^{rig}}X(v)$. If
$v<\frac{p-2}{2p^2-p}$, we have a diagram
\[\xymatrix{
&C_i(v)\ar[ld]_{p_1}\ar[rd]^{p_2}&\\
X(v)& &X(v).
 }\]
As in \cite{AIP} 5.6, for $v<\frac{1}{2p^{m-1}}$ (resp. $v<\frac{1}{3^m}$ if $p=3$) and $\ul{w}=(w_{k,j})_{1\leq j\leq k\leq n}\in
]\frac{v}{p-1},m-v\frac{p^m}{p-1}]^{\frac{n(n+1)}{2}}$ satisfying $w_{k+1,j}\geq w_{k,j}, w_{k,j-1}\geq w_{k,j}$, we can introduce a space $\IW_{\ul{w}}^{0+}$ over $X(v)$ such that for $\ul{w}$ with all $w_{k,j}=w$, $\IW_{\ul{w}}^{0+}=\IW_{w}^{0+}$. Let $\pi: A\ra A'$ be the universal isogeny over $C_i(v)$. We have a map $\pi^\ast: \omega_{A'}\ra\omega_A$. It induces a map $\wt{\pi}^\ast: p_2^\ast\T_\tau^{rig}\ra p_1^\ast\T_\tau^{rig}$ which sends a basis $v_1',\dots,v_n'$ of $\omega_{A',\tau}$ to $p^{-1}\pi^\ast v_1',\dots,p^{-1}\pi^{\ast}v_{n-i}',\pi^{\ast}v_{n-i+1}',\dots,\pi^\ast v_n'$. This map is an isomorphism. It induces an isomorphism $\wt{\pi}^{\ast-1}: p_1^\ast\T_\tau^{rig}/U_\tau \ra p_2^\ast\T_\tau^{rig}/U_\tau$. We have the following proposition.
\begin{proposition}
Let $v<\inf\{\frac{1}{3p^{m-1}},\frac{p-2}{2p^2-p}\}$ and $\ul{w}=(w_{k,j})_{1\leq j\leq k\leq n-1}$ with $w_{k,j}\in ]\frac{v}{p-1},m-2-v\frac{p^m}{p-1}]$. Then the isogeny $\pi$ induces $\wt{\pi}^{\ast-1}p_1^\ast\IW_{\ul{w}}^{0+}\subset
p_2^\ast\IW_{\ul{w'}}^{0+}$ where $w_{k,j}'=w_{k,j}$ for $j\leq k\leq i$, $w_{k,j}'=1+w_{k,j}$ for $j\leq i$ and $k\geq i+1$, and $w_{k,j}'=w_{k,j}$ for $j\geq i+1$.
\end{proposition}
\begin{proof}Identical to the proof of Proposition 6.2.2.2. in \cite{AIP}.
\end{proof}
By the above proposition we have
\[\wt{\pi}^{\ast-1}p_1^\ast(\IW_{\ul{w}}^{0+}\times\T^{\tau,rig})\subset
p_2^\ast(\IW_{\ul{w'}}^{0+}\times\T^{\tau,rig}).\]
Let $\ul{w}$ and $\ul{w}'$ be as above, $\kappa$ be a $\inf\{w_{j,j}\}$-analytic character. Then we have also the sheaves $\omega_{\ul{w}}^{\dag\kappa}$ and $\omega_{\ul{w}}^{\dag\kappa}$ on $X(v)$. For $2\leq i\leq n-1$ we define the Hecke operator $U_i$ as the composition:
\[H^0(X(v),\omega_{\ul{w'}}^{\dag\kappa})\st{p_2^\ast}{\ra}H^0(C_i(v),p_2^\ast\omega_{\ul{w'}}^{\dag\kappa})
\st{\wt{\pi}^{\ast}}{\ra}H^0(C_i(v),p_1^\ast\omega_{\ul{w}}^{\dag\kappa})\st{p^{-in}Trp_1}{\longrightarrow}
H^0(X(v),\omega_{\ul{w}}^{\dag\kappa}).\]
By abuse of notation, we also denote by $U_i$ the endomorphism of $H^0(X(v),\omega_{\ul{w}}^{\dag\kappa})$ obtained by composing the above operator with the restriction map $H^0(X(v),\omega_{\ul{w}}^{\dag\kappa})\ra H^0(X(v),\omega_{\ul{w'}}^{\dag\kappa})$.

Let $m\geq 1$ be some large integer. For $v<\inf\{\frac{1}{3p^{m-1}},\frac{p-2}{2p^2-p}\}$, $w\in ]\frac{v}{p-1},m-n-v\frac{p^m}{p-1}]$, and $\kappa$ a $w$-analytic character, the product $\prod_{i=1}^{n-1}U_i$ induces a map $H^0(X(\frac{v}{p}),\omega_{\ul{w'}}^{\dag\kappa})\ra H^0(X(v),\omega_{w}^{\dag\kappa})$ for some suitable
$\ul{w'}$ with $w_{k,j}'\geq w$. Let $U$ be the composition of $\prod_{i=1}^{n-1}U_i$ with the restriction map $H^0(X(v),\omega_{w}^{\dag\kappa})\ra H^0(X(\frac{v}{p}),\omega_{\ul{w'}}^{\dag\kappa})$. Then we have the following proposition.
\begin{proposition}$U$ is a compact endomorphism of $M_w^{\dag\kappa}(X(v))=H^0(X(v),\omega_{w}^{\dag\kappa})$.
\end{proposition}
\begin{proof}The operator $U_1$ improves the radius of overconvergence, and for $2\leq i\leq r$ the operators $U_i$ improve analyticity, therefore all of them are compact. Moreover, the natural restriction map is compact. Hence their composition $U$ is compact.
\end{proof}

Let $\mathbb{T}_p$ be the algebra generated by $U_1,\dots,U_{n-1}$, and $\mathbb{T}^{K^p}$ be the commutative Hecke algebra defined in the end of the last subsection. For the above $v, w, \kappa$, we have the morphism of algebras
\[\mathbb{T}^{K^p}\otimes\mathbb{T}_p\ra End(M^{\dag\kappa}_w(X(v))).\]
As explained in Remark 4.1, although the different conventions on the correspondences $C_1,\dots,C_{n-1}$ may produce slightly different algebras $\mathbb{T}_p$, when passing to the image of the Hecke algebra $\mathbb{T}^{K^p}\otimes\mathbb{T}_p$ into $End(M^{\dag\kappa}_w(X(v)))$, these differences disappear.

\section{Analytic continuation and classicality}
Recall that we have fixed the weights $(\kappa_\sigma)_{\sigma\neq\tau}=(k_{1\sigma},\dots,k_{n\sigma})_{\sigma\neq\tau}\in\prod_{\sigma\neq\tau}\Z^n_+$.
Let $\kappa=\kappa_\tau=(k_1,\dots,k_{n-1},k_n)\in \Z^{n-1}_+\times\Z$ be $w$-analytic for some rational number $0<w\leq 1$. Then we have the inclusions
\[H^0(X,\omega^\kappa)\hookrightarrow H^0(X(v),\omega^\kappa)\hookrightarrow H^0(X(v),\omega_w^{\dag\kappa}).\]We will establish a criterion for an element in $H^0(X(v),\omega_w^{\dag\kappa})$ to be classical, i.e. in the image of $H^0(X,\omega^\kappa)$. For $\underline{a}=(a_1,\dots,a_{n-1})\in [0,+\infty]^{n-1}$, we set $M^{\dag\kappa}_w(X(v))^{<\underline{a}}$ for the union of the generalized eigenspaces where $U_{i}$ has finite slope $<a_i$ for $1\leq i\leq n-1$.
\begin{theorem}Let $\kappa=(k_1,\dots,k_{n-1},k_n)\in \Z^{n-1}_+\times\Z$ be $w$-analytic for some rational number $0<w\leq 1$.
Let $\underline{a}=(a_1,\dots,a_{n-1})\in\R^{n-1}_{\geq 0}$ with $a_i=k_{n-i}-k_{n-i+1}+1$ when $2\leq i\leq n-1$ and $a_1=k_n+k_{n-1}-n+1$. Then we have the inclusion
\[M_w^{\dag\kappa}(X(v))^{<\underline{a}}\subset H^0(X,\omega^\kappa).\]
\end{theorem}
The proof of this theorem consists of the following two propositions.

\begin{proposition}
Let $\underline{a'}=(+\infty,a_2,\dots,a_{n-1})\in [0,+\infty]^{n-1}$ with $a_i$ the same as in the above theorem for $2\leq i\leq n-1$. Then we have the inclusion
\[M_w^{\dag\kappa}(X(v))^{<\underline{a'}}\subset H^0(X(v),\omega^\kappa).\]
\end{proposition}
\begin{proof}
We argue as \cite{AIP} 7.2 and 7.3. First, there is a exact sequence of sheaves over $X(v)$
\[0\ra \omega^\kappa\st{d_0}{\longrightarrow}\omega_w^{\dag\kappa}\st{d_1}{\longrightarrow}\bigoplus_{\alpha\in \Delta}\omega_w^{\dag s_\alpha\cdot\kappa}.\] Here $\Delta$ is the set of simple roots of $GL_{n-1}\times GL_1$, $s_\alpha$ is the associated element in the Weyl group to $\alpha\in\Delta$, $s_\alpha\cdot\kappa$ is the usual dot action, $d_1=\oplus \Theta_\alpha: \omega_w^\kappa\ra \bigoplus_{\alpha\in \Delta}\omega_w^{\dag s_\alpha\cdot\kappa}$ is the sum of the maps $\Theta_\alpha$ which are defined as in subsection 7.2 of \cite{AIP}. Next, for $2\leq i\leq n-1$ we have the operators $\delta_i$ for $V_{\kappa'_\tau}^{w-an}$ as defined in loc. cit. 2.5 for $d_i=
 \left(\begin{array}{ccc}
 p^{-1}1_{n-i} & & \\
& 1_{i-1}&\\
& &1
  \end{array}\right)
\in GL_{n-1}(\Q_p)\times GL_1(\Q_p)$. $\delta_i$ extends naturally to an operator on $V_{\kappa'_\tau}^{w-an}\otimes V^\tau$ (by acting trivially on $V^\tau$, where $V^\tau$ is as in Prop. 3.7). Then one checks that locally on the fibers, the operators $U_i$ and $\delta_i$ are compatible for $2\leq i\leq n-1$. Using this we get
\[U_i\Theta_\alpha(f)=\alpha(d_{i})^{<\kappa,\alpha^\vee>+1}\Theta_\alpha U_i(f)\]from the corresponding equalities for the representation $V^{w-an}_{\kappa'_\tau}$ (subsection 2.5 of \cite{AIP}). One then uses the slope condition as the proof of loc. cit. Prop. 7.3.1 to conclude.
\end{proof}

\begin{proposition}
Let $H^0(X(v),\omega^\kappa)^{<k_n+k_{n-1}-n+1}$ be the eigenspace where $U_1$ has slope $<k_n+k_{n-1}-n+1$. Then
it is included in the space of classical forms
\[H^0(X(v),\omega^\kappa)^{<k_n+k_{n-1}-n+1}\subset H^0(X,\omega^\kappa).\]
\end{proposition}

We begin the preparation for the proof of the above proposition. It will be finished in Prop. 5.7. The method which we use here is by studying the analytic continuation of finite slope overconvergent eigenforms as in \cite{PS}. There is a related work \cite{Bi} of Bijakowski, where classicality results for modular forms over some general PEL type Shimura (with unramified local reductive groups) are proved. Although Bijakowski also used the method of analytic continuation to prove classicality results, there are still many differences between our approach below and that in \cite{Bi}. Namely, Bijakowski studied intensively the geometry near the region with integer degrees in the rigid analytic Shimura varieties, while we are mainly based on the geometry of the special fiber, which is simpler in our special case. In this sense our method is more close to that of \cite{PS}. We remark that there is a related result of Johansson in \cite{J} by studying the rigid cohomology for a more restricted sub class of our Shimura varieties.

We define a function of $deg$ over $X^{rig}$ as follows
\[\begin{split}deg: X^{rig}&\longrightarrow [0,1]\\x&\mapsto deg(\tr{Fil}_1H_x[p]).\end{split}\]
Then it is continuous, and note that $X(v)=deg^{-1}([1-v,1]), ]X_1^0[=deg^{-1}(1)=X(0)$. For any point $y\in U_1(x)$, we have $deg(y)\geq deg(x)$. Moreover, the equality holds if and only if $\tr{Fil}_1H_x[p]$ is a truncated $p$-divisible group of level one, see \cite{PS} Propositions 3.1.1 and 3.1.2. Therefore any finite slope eigenform for $U_1$ in $H^0(X(v),\omega^\kappa)$ extends to the analytic domain $deg^{-1}(]0,1])$. We would like to extend them further to the whole space $X^{rig}$. Since in our special case
\[H^0(X^{rig},\omega^\kappa)=H^0(]\ov{X}^{n-1}[\coprod]\ov{X}^{n-2}[,\omega^\kappa),\]
it suffices to extend these forms to the tube $]\ov{X}^{n-1}[\coprod]\ov{X}^{n-2}[$. We know that we can do this for the domain
\[deg^{-1}(]0,1])\cap(]\ov{X}^{n-1}[\coprod]\ov{X}^{n-2}[).\]
To extend to the remaining domain
\[deg^{-1}(0)\cap](\ov{X}^{n-1}[\coprod]\ov{X}^{n-2}[),\]we will use the condition that the slope for $U_1$ satisfies
$v(\alpha)<k_n+k_{n-1}-n+1$ to define some Kassaei series and then glue.

Consider a strict neighborhood $]\ov{X}^{n-1}[_\varepsilon$ of $]\ov{X}^{n-1}[$ such that the canonical subgroup $C\subset H[p]$ inside the universal $p$-divisible group $H$ exists. Then the correspondence $C_1^{rig}$ has the decomposition
\[C_1^{rig}\cap p_1^{-1}(]\ov{X}^{n-1}[_\varepsilon)=C_1^<\sqcup C_1^\geq\]when restricts to $]\ov{X}^{n-1}[_\varepsilon$, where $C_1^<$ parameterizes the subgroups $L\subset H[p]$ which has trivial intersection with $C$, and $C_1^\geq$ parameterizes the subgroups $L\subset H[p]$ which contains $C$. Then the degree of $L$ is $<1$ over $C_1^<$ while $\geq 1-\eta$ over $C_1^\geq$, where $0\leq \eta <1$ is some real number which depends on $\varepsilon$. Over the tube $]\ov{X}^{n-2}[$, we have also a decomposition
\[C_1^{rig}\cap p_1^{-1}(]\ov{X}^{n-2}[)=C_1^<\sqcup C_1^\geq,\] where $C_1^<$ parameterizes the subgroups $L\subset H[p]$ such that $L^{et}=H[p]^{et}$ which have height $n-2$, and $C_1^\geq$ parameterizes the subgroups $L\subset H[p]$ such that $L^{et}\subsetneq H[p]^{et}$ and thus the local parts $L^{0}=H[p]^{0}$. The degree of $L$ is $<1$ over $C_1^<$ while $=1$ over $C_1^\geq$. Here we have used the fact that, over the formal completion of $S_K$ along $\ov{X}^{n-2}$ there is an exact sequence of finite locally free groups
\[0\longrightarrow H[p]^0\longrightarrow H[p]\longrightarrow H[p]^{et}\longrightarrow 0,\]
which induces an exact sequence for the subgroup $L\subset H[p]$
\[0\longrightarrow L^0\longrightarrow L\longrightarrow L^{et}\longrightarrow 0,\]see \cite{PS} Lemme 5.1.2.
Therefore, we obtain decompositions of the Hecke operator $U_1$
\[U_1=U_1^<\sqcup U_1^\geq\]
over $]\ov{X}^{n-1}[_\varepsilon$ and $]\ov{X}^{n-2}[$. If we consider $U_1^<, U_1^\geq$ as maps on the set of admissible opens (defined by the correspondences $C_1^<, C_1^\geq$ respectively), then by construction \[U_1^<(]\ov{X}^{n-1}[_\varepsilon)\subset deg^{-1}(]0,1]), U_1^<(]\ov{X}^{n-2}[)\subset deg^{-1}(]0,1]).\]
\begin{remark}
Here we use our special signature condition $(1,n-1)\times (0,n)\times\cdots\times (0,n)$ to have the decompositions of the correspondence $C_1^{rig}$ over $]\ov{X}^{n-1}[_\varepsilon$ and $]\ov{X}^{n-2}[$. In general, one should work on the tube over each Kottwitz-Rapoport strata to get the decompositions of the related Hecke correspondence, see \cite{PS} section 5.
\end{remark}

Let $f\in H^0(X(v),\omega^\kappa)$ be an overconvergent eigenform for $U_1$ with eigenvalue $\alpha\neq 0$. We know that $f$ extends to a form over $deg^{-1}(]0,1])\cap (]\ov{X}^{n-1}[\coprod]\ov{X}^{n-2}[)$, which we still denote by $f$. Assume that $v(\alpha)<k_n+k_{n-1}-n+1$. We claim that there exists an overconvergent form $\tilde{f}$ over $]\ov{X}^{n-1}[\coprod]\ov{X}^{n-2}[$ which extends $f$.

Let $\varepsilon$ be a real number such that $k_n+k_{n-1}-n+1-v(\alpha)>\varepsilon>0$. We choose an $\varepsilon_0>0$ such that over any admissible open $V\subset ]\ov{X}^{n-1}[_{\varepsilon_0}$, the norm of $U_1^\geq$ has the estimate (cf. \cite{PS} Lemme. 5.4.6)
\[\|U_1^\geq\|_V\leq p^{-k_n-k_{n-1}+n-1+\varepsilon},\]
where for an operator $T: H^0(T(V), \Fm)\ra H^0(V,\Fm)$ ($\Fm$ is a locally free sheaf equipped with a norm $|\cdot|$), its norm over $V$ is defined as
\[\|T\|_V=\inf\{\beta\in\R^+|\,|T(f)|_V\leq \beta|f|_{T(V)},\,\forall\,f\in H^0(T(V),\Fm)\}.\]
Let $\varepsilon_k=(\frac{1}{N})^k\varepsilon_0$ for all $k\geq 1$. Here $N\geq 1$ is such that $U_1^\geq(]\ov{X}^{n-1}[_{\varepsilon_0})\subset ]\ov{X}^{n-1}[_{N\varepsilon_0}$ (loc. cit. Prop. 5.4.5).
For $k\geq 1$, $f$ as above, we define
\[f_k=\sum_{j=0}^{k-1}\alpha^{-j-1}((U_1^\geq)^j\circ U_1^<)(f),\]
which is a well defined element in $H^0(]\ov{X}^{n-1}[_{\varepsilon_k},\omega^\kappa)$.
We have the following proposition.
\begin{proposition}
For all $k\geq 1$, we have the estimate
\[|f_{k+1}-f_k|_{]\ov{X}^{n-1}[_{\varepsilon_{k+1}}}= O(p^{-k(k_n+k_{n-1}-n+1-v(\alpha)-\varepsilon)}).\]
\end{proposition}
\begin{proof}
By definition over $]\ov{X}^{n-1}[_{\varepsilon_{k+1}}$
\[f_{k+1}-f_k=\alpha^{-k-1}((U^\geq_1)^k\circ U_1^<)(f).\]
Since the norm of $U^<_1$ is always bounded, we can conclude by the estimate for the norm of $U_1^\geq$.
\end{proof}

Over the tube $]\ov{X}^{n-2}[$ we have also the following definition.
For $k\geq 1$, $f$ as above, we define
\[f_k'=\sum_{j=0}^{k-1}\alpha^{-j-1}((U_1^\geq)^j\circ U_1^<)(f),\]
which is a well defined element in $H^0(]\ov{X}^{n-2}[,\omega^\kappa)$.
For any admissible open $V\subset ]\ov{X}^{n-2}[$ we have the following norm estimate
\[\|U_1^\geq\|_V\leq p^{-k_n-k_{n-1}+n-1}.\]
Similar to the above, we have the proposition.
\begin{proposition}
For all $k\geq 1$, we have the estimate
\[|f_{k+1}'-f_k'|_{]\ov{X}^{n-2}[}= O(p^{-k(k_n+k_{n-1}-n+1-v(\alpha))}).\]
\end{proposition}

Let $f\in H^0(X(v),\omega^\kappa)$ be as above. Now we can prove the proposition, which will finish the analytic continuation of $f$ to the whole space $X^{rig}$ under the slope condition. Thus it is classical.
\begin{proposition}
There exists a unique section of $\omega^\kappa$ over
\[]\ov{X}^{n-1}[\coprod ]\ov{X}^{n-2}[\]
which extends $f$.
\end{proposition}
\begin{proof}
We consider the larger domain
\[]\ov{X}^{n-1}[_{\varepsilon_k}\bigcup]\ov{X}^{n-2}[\] for all $k\geq 1$. By the above we can define $f_k$ and $f_k'$ over $]\ov{X}^{n-1}[_{\varepsilon_k}$ and $]\ov{X}^{n-2}[$ respectively. Let $V=]\ov{X}^{n-1}[_{\varepsilon_k}\bigcap
]\ov{X}^{n-2}[$. We study the norm of $f_k-f_k'$ over $V$. First we consider $V$ as a sub domain of $]\ov{X}^{n-1}[_{\varepsilon_k}$. Then the operator $U_1$ has a decomposition
\[U_1=U_1^\geq\sqcup U_1^<\] according to the decomposition of the correspondence $C_1^{rig}$ when restricting to $p_1^{-1}(V)$. Now we consider $V$ as a sub domain of $]\ov{X}^{n-2}[$, the operator $U_1^\geq$ has a further decomposition
\[U_1^\geq=U_1'\sqcup U_1'',\]
where the part $U_1'$ corresponds to the sum over the subgroups $L$ of $H[p]$ such that $L\cap \tr{Fil}_1H[p]=0$ and $L^{et}=H[p]^{et}$ (which have height $n-2$), and the part $U_1''$ corresponds to the sum over the rest subgroups $L$ of $H[p]$ such that $L\cap \tr{Fil}_1H[p]=0$ and $L^{et}\subsetneq H[p]^{et}$ (which have degree 1). Therefore, the image of $U_1'$ is included in $deg^{-1}(]0,1])$. We have
\[f_k-f_k'=-\alpha^k(\sum_{j=1}^k(U_1'')^{k-j}\circ(U_1')^j)(f).\]
By using similar estimate for the norms of the operators $U_1'$ and $U_1''$, we have
\[|f_k-f_k'|_V=O(p^{-k(k_n+k_{n-1}-n+1-v(\alpha)-\varepsilon)}).\]
This will suffice to prove the proposition by using Kassaei's gluing lemma in \cite{Ka1}.
\end{proof}

\section{Families of overconvergent automorphic forms}

Let $\mathcal{U}=Sp(A)\subset\W$ be any affinoid open subset. Then there exists a $w_{\mathcal{U}}>0$ such that the universal character $\kappa^{un}: T_\tau(\Z_p)\times\W\ra\Cm_p^\times$ restricted to $\mathcal{U}$ extends to an analytic character $\kappa^{un}: T_\tau(\Z_p)(1+p^{w_{\mathcal{U}}}O_{\Cm_p})\times\mathcal{U}\ra\Cm_p^\times$, cf. \cite{AIP} Prop. 2.2.2. Let $m\in \mathbb{N}, v\leq\frac{1}{2p^{m-1}}$ (resp. $\frac{1}{3p^{m-1}}$ for $p=3$) and $w\in ]m-1+\frac{v}{p-1},m-v\frac{p^m}{p-1}]$ satisfying $w\geq w_\mathcal{U}$. We have the following proposition which says the construction in section 3 works in families, see also loc. cit. Prop. 8.1.1.1.
\begin{proposition}
There exists a sheaf $\omega_w^{\dag\kappa^{un}}$ on $X(v)\times\mathcal{U}$ such that for any weight $\kappa\in\mathcal{U}$, the fiber of $\omega_w^{\dag\kappa^{un}}$ over $X(v)\times\kappa$ is $\omega_w^{\dag\kappa}$.
\end{proposition}
\begin{proof}
Consider the projection $\pi\times1: \IW_w^{0+}\times\mathcal{U}\ra X(v)\times\mathcal{U}$. Recall we have the bundle $\omega^\tau=\otimes_{\sigma\neq\tau}\omega_\sigma^{\kappa_\sigma}$ over $X(v)$ for the fixed weight $(\kappa_\sigma)_{\sigma\neq\tau}$ apart from $\tau$. We take $\omega_w^{\dag\kappa^{un}}=(\pi\times1)_\ast\mathcal{O}_{\IW_w^{0+}\times\mathcal{U}}[(\kappa^{un})']
\otimes\omega^\tau$.
\end{proof}

Let $M_{v,w}$ be the Banach $A$-module $H^0(X(v)\times\mathcal{U},\omega_w^{\dag\kappa^{un}})$. Then similar to section 4, there is an action of the Hecke algebra $\mathbb{T}^{K^p}\otimes\mathbb{T}_p$ on $M_{v,w}$ for $v$ small enough.  For $w'\in\Q_{>0}$,  as $\W(w')$ is affinoid, we can consider the admissible open subspace $\mathcal{U}\subset\W$ in the form $\mathcal{U}=\W(w')=Sp(A)$. In this case, we can take $w=w'$.
\begin{proposition}
\begin{enumerate}
\item The Banach $A$-module $M_{v,w}$ is projective.
\item For any $\kappa\in\mathcal{U}$, the specialization map
\[M_{v,w}\ra H^0(X(v),\omega_w^{\dag\kappa})\] is surjective.
\end{enumerate}
\end{proposition}
\begin{proof}
The proof is similar to the proof of Prop. 8.2.3.3 in \cite{AIP}, except that our Shimura varieties are proper, so the things here are simpler. First, note that $\mathfrak{m}_w^{\dag\kappa^{un}}$ over $\mathfrak{X}(v)\times \mathfrak{W}(w)$ is a formal Banach sheaf, where $\mathfrak{W}(w)$ is a formal model of $\W(w)$, and $\mathfrak{m}_w^{\dag\kappa^{un}}$ is constructed in the same way as $\mathfrak{m}_w^{\dag\kappa}$, which is a formal model of the sheaf $\omega_w^{\dag\kappa^{un}}$ in the above proposition. Next, since $X(v)$ is affinoid, we can take a finite affine covering $\mathfrak{U}=(\mathfrak{V}_j)_{1\leq j\leq k}$ of $\mathfrak{X}(v)$. Let $\ul{i}=(i_1,\dots,i_j)$ be an index with $1\leq i_1<\cdots<i_j\leq k$, and $\mathfrak{V}_{\ul{i}}$ be the intersection of $\mathfrak{V}_{i_l}$ for $1\leq l\leq j$ which is again an affine formal scheme. Let $M_{\ul{i}}=H^0(\mathfrak{V}_{\ul{i}}\times\mathfrak{W}(w),\mathfrak{m}_w^{\dag\kappa^{un}})$, then it is isomorphic to the $p$-adic completion of a free $B$-module, where $\mathfrak{W}(w)=Spf(B)$ (so $A=B[\frac{1}{p}]$). Let $M=M_{v,w}=H^0(\mathfrak{X}(v)\times \mathfrak{W}(w),\mathfrak{m}_w^{\kappa^{un}})[\frac{1}{p}]$. Since the rigid analytic fiber of the covering $\mathfrak{U}$ forms a covering of $X(v)$ and $X(v)\times\W(w)$ is affinoid, we get a resolution of $M$ by the projective modules $M_{\ul{i}}[\frac{1}{p}]$. As a result $M$ is projective. Finally, the specialization map is surjective by considering the Koszul resolution of $A/m_\kappa$ and the double complex obtained by taking the Cech complex of the Koszul complex, see the proof of Corollary 8.2.3.2 of \cite{AIP}.

\end{proof}

By the above proposition, one can apply Coleman's spectral theory as in \cite{AIP} 8.1.2 or in \cite{Bu} 5.7 to construct an equidimensional eigenvariety over $\W$. More precisely we have the following theorem. In the following we denote $\mathcal{O}(\mathcal{E})^{red}$ as the reduced algebra associated to $\mathcal{O}(\mathcal{E})$.
\begin{theorem}
There is a rigid analytic variety $\mathcal{E}$ over $L$ and a locally finite map to the weight space $w:\mathcal{E}\ra\W$, such that
\begin{enumerate}
\item $\mathcal{E}$ is equidimensional of dimension $n$.
\item We have a character $\Theta: \mathbb{T}^{K^p}\otimes\mathbb{T}_p\ra \mathcal{O}(\mathcal{E})$. For any $\kappa\in\W$, $w^{-1}(\kappa)$ is in bijection with the eigensystems of $\mathbb{T}^{K^p}\otimes\mathbb{T}_p$ acting on the space of finite slope locally analytic overconvergent automorphic forms of weight $\kappa$.
\item For any $\kappa=(k_1,\dots,k_{n-1},k_n)\in \Z^{n-1}_+\times\Z\subset\W$, if $x\in w^{-1}(\kappa)$ satisfies  $v(\Theta_x(U_i))<k_{n-i}-k_{n-i+1}+1$ for $2\leq i\leq n-1$ and $v(\Theta_x(U_{1}))<k_n+k_{n-1}-n+1$, then the character $\Theta_x$ comes from a weight $\kappa$ automorphic eigenform on $X$. Here $\Theta_x$ is the composition of $\Theta$ with the evaluation map $ev_x: \mathcal{O}(\mathcal{E})\ra k(x)$ ($k(x)$ is the residue field of $x$).
\item There is a Galois pseudo-character $T: Gal(\ov{\Q}/F)\ra \mathcal{O}(\mathcal{E})^{red}$, such that for any point $x\in \mathcal{E}$, there is a continuous semi-simple representation $\rho_x: Gal(\ov{\Q}/F)\ra GL_n(\overline{k(x)})$ and the trace of this Galois representation is $T_x$. Here $T_x$ is the composition of $T$ with the evaluation map $ev_x: \mathcal{O}(\mathcal{E})^{red}\ra k(x)$.
\end{enumerate}
\end{theorem}
\begin{proof}
(1) and (2) come from the construction. (3) was proved in section 5. For (4), we use the density of classical points as proved in (3) and the results of Harris-Taylor in \cite{HT} to get the desired Galois pseudo-character, as in \cite{BC} 7.5.2. We sketch the proof as follows. Let $Z\subset\mathcal{E}$ be the subset defined by (3), then it is Zariski dense in $\mathcal{E}$. For any point $x\in Z$, by the main result (Theorem C) of \cite{HT} we have a Galois representation $\rho_x: Gal(\ov{\Q}/F)\ra GL_n(\ov{\Q}_p)$, such that $Tr\rho_x(Frob_l)=\Theta_x(T_l)$ for any $l\neq p$ and such that $K_l$ is hyperspecial. Here $Frob_l$ is the geometric Frobenius at $l$. Consider the map
\[\prod_{x\in Z}Tr\rho_x: Gal(\ov{\Q}/F)\ra \prod_{x\in Z}\ov{\Q}_p.\]Since $Z$ is Zariski dense in $\mathcal{E}$, the map $f\mapsto (f(x))_{x\in Z}$ induces a closed immersion \[(\mathcal{O}(\mathcal{E})^{red})^{\leq 1}\hookrightarrow \prod_{x\in Z}\ov{\Q}_p.\] Here the upper subscript $\leq 1$ means that we consider the subset of functions with norm $\leq 1$. On $\prod_{x\in Z}\ov{\Q}_p$ we take the product topology for the $p$-adic topology on $\ov{\Q}_p$. The statement that the image of $(\mathcal{O}(\mathcal{E})^{red})^{\leq 1}$ in $\prod_{x\in Z}\ov{\Q}_p$ is closed, comes from the fact that $(\mathcal{O}(\mathcal{E})^{red})^{\leq 1}$ is compact, see Lemma 7.2.10 of \cite{BC}. Moreover, since the prime to $p$ Hecke operators $T_l$ have norm $\leq 1$, we see that the elements $\ov{\Theta(T_l)}$ (image of $\Theta(T_l)$ in $\mathcal{O}(\mathcal{E})^{red}$) are included in $(\mathcal{O}(\mathcal{E})^{red})^{\leq 1}$. Therefore by Cebotarev's theorem we can conclude that $\prod_{x\in Z}Tr\rho_x$ factors through $(\mathcal{O}(\mathcal{E})^{red})^{\leq 1}$. Thus the desired Galois pseudo-character $T$ exists (for the fact that $T_x=Tr\rho_x$ which takes values in $k(x)$, see p.172 of \cite{BC}).
\end{proof}
\begin{remark}
It will be interesting to study the properties of the Galois representations $\rho_x$ for $x\in \mathcal{E}$. For example, when is it crystalline or potentially semi-stable? By the main results of \cite{HT} and \cite{Shin}, it is true that the crystalline points (that are the points $x\in\mathcal{E}$ such that $\rho_x$ are crystalline) are dense in $\mathcal{E}$, see \cite{BC} 7.6. Also it will be interesting to find applications of our eigenvarieties to the construction of Galois representations as in \cite{Ch2}.
\end{remark}
\begin{remark}
In this paper we worked out the construction of some equidimensional eigenvarieties for the Shimura varieties studied by Harris-Taylor with local reflex field $\Q_p$. An immediate extension to the Shimura varieties studied in \cite{Shin} works. One may try to work in the general case that the local reflex field is not necessary $\Q_p$. However, for this one will need a theory of canonical subgroups for $\varpi$-divisible $O$-modules in the sense of section II.1 of \cite{HT} ($O$ is the integer ring of some finite extension of $\Q_p$, $\varpi$ is a fixed uniformizer of $O$). Here one should consider the Faltings dual instead of the usual Cartier-Serre dual, and the expected canonical subgroups should be strict in the sense of \cite{F}. This theory should be as good as the theory of canonical subgroups for usual $p$-divisible groups as in \cite{F2} (the case for dimension one may be easier, see \cite{Ka} and \cite{Br} for the case of height two). Since the first version of this article appeared, recently in a joint work \cite{ST} with Tian, we have developed a such theory and applied it to the construction of eigenvarieties in some more general case along the same line as here. For example, we can treat the case of unitary Shimura varieties with signature $(d,n-d)\times(0,n)\times\cdots\times(0,n)$. For $p$-divisible groups with more general additional structures induced from more general PEL type Shimrua varieties, one will need a theory of more generalized canonical subgroups. Once such a theory is available, combined with the theory of arithmetic compactification of PEL type Shimura varieties (with good or bad reductions, see \cite{L}), one can deal with any PEL type Shimura varieties (compact or not) with the usual ordinary loci empty or not to construct geometrically the associated eigenvarieties.
\end{remark}

In the future work we will compare the eigenvariety $\mathcal{E}$ with those introduced by Emerton in \cite{E} and Chenevier in \cite{Ch1}. In particular we will study the completed cohomology of these Shimura varieties, and the possible $p$-adic Jacquet-Langlands correspondence for $\mathcal{E}$ and the corresponding eigenvariety constructed in \cite{Ch1} as in the case of curves studied in \cite{Ch0}. See \cite{N} for some results in this direction in the case of Shimura curves.


\begin{thebibliography}{99}
\bibitem{AIP}F. Andreatta, A. Iovita, V. Pilloni, \textsl{$p$-adic families of Siegel modular cuspforms}, Ann. Math. 181 (2) (2015), 623-697.
\bibitem{AIP2}F. Andreatta, A. Iovita, V. Pilloni, \textsl{$p$-adic families of Hilbert modular forms}, Preprint, available at http://perso.ens-lyon.fr/vincent.pilloni/AIP2.pdf
\bibitem{BC}J. Bella\"{\i}che, G. Chenevier, \textsl{Families of Galois representations and Selmer groups}, Ast\'erisque 324, Soc. Math. France, 2009.
\bibitem{Bi}S. Bijakowski, \textsl{Classicit\'e de formes modulaires surconvergentes}, Prenprint, arXiv:1212.2035.
\bibitem{Br}R. Brasca, \textsl{$p$-adic modular forms of non-integral weight over Shimura curves}, Compos. Math. 149 (2013), no.1, 32-62.
\bibitem{Br1}R. Brasca, \textsl{Eigenvarieties for cuspforms over PEL type Shimura varieties with dense ordinary locus}, Preprint, arXiv:1407.7973.
\bibitem{Bu}K. Buzzard, \textsl{Eigenvarieties}, Proceedings of the LMS Durham conference on $L$-functions and arithmetic geometry, London Math. Soc. Lecture Note Ser. 320 (2007), 59-120, Cambridge Univ. Press, Cambridge.
\bibitem{Ch0}G. Chenevier, \textsl{Une correspondance de Jacquet-Langlands $p$-adique}, Duke Math. J. 126 (2005), no.1, 161-194.
\bibitem{Ch1}G. Chenevier, \textsl{Une application des vari\'et\'es de Hecke des groupes unitaires}, Preprint, available at http://gaetan.chenevier.perso.math.cnrs.fr/articles/famgal.pdf
\bibitem{Ch2}G. Chenevier, M. Harris, \textsl{Constructions of Galois representations II}, Cambridge Journal of Mathematics 1 (2013), 57-73.
\bibitem{CM}R. Coleman, B. Mazur, \textsl{The eigencurve}, In Galois representations in arithmetic algebraic geometry (Durham, 1996), London Math. Soc. Lecture Note Ser. 254 (1998), 1-113, Cambridge Univ. Press, Cambridge.
\bibitem{D} Y. Ding, \textsl{Formes modulaires $p$-adiques sur les courbes de Shimura unitaires et compatibilit\'e local-global}, PhD thesis of Universit\'e Paris-Sud 11, 2015.
\bibitem{E}M. Emerton, \textsl{On the interpolation of systems of eigenvalues attached to automorphic Hecke eigenforms}, Invent. Math. 164 (2006), 1-84.
\bibitem{F}G. Faltings, \textsl{Group schemes with strict $\mathcal{O}$-action}, Mos. Math. J. (2002), Vol.2, no.2, 249-279.
\bibitem{F2} L. Fargues, \textsl{La filtration canonique des points de torsion des groupes $p$-divisibles},
Ann. Sci. \'Ecole Norm. Sup. 44 (2011), 905-961.
\bibitem{H}M. Harris, \textsl{The local Langlands correspondence: Notes of (half) a course at the IHP spring 2000}, Ast\'erisque 298 (2005), 17-145.
\bibitem{HT}M. Harris, R. Taylor, \textsl{The geometry and cohomology of some simple Shimura varieties}, Ann.
of Math. Stud. 151, Princeton Univ. Press, 2001.
\bibitem{J} C. Johansson, \textsl{A trace formula approach to control theorems for overconvergent automorphic forms}, Preprint, arXiv:1309.7239.
\bibitem{Ka}P.L. Kassaei, \textsl{$\mathcal{P}$-adic modular forms over Shimura curves over totally real fields}, Compos. Math. 140 (2004), no. 2, 359-395.
\bibitem{Ka1}P.L. Kassaei, \textsl{A gluing lemma and overconvergent modular forms}, Duke Math. J. 132 (2006), no.3, 509-529.
\bibitem{L} K. Lan, \textsl{Compactification of PEL type Shimura varieties and Kuga families}, Preprint, available at http://www.math.umn.edu/$\sim$kwlan/articles/cpt-ram-ord.pdf
\bibitem{N}J. Newton, \textsl{Completed cohomology of Shimura curves and $p$-adic Jacquet-Langlands correspondence}, Math. Annalen 355 (2013), no. 2, 729-763.
\bibitem{PS}V. Pilloni, B. Stroh, \textsl{Surconvergence et classicit\'e: le cas d\'eploy\'e}, Preprint, available at http://perso.ens-lyon.fr/vincent.pilloni/surconv-deploye.pdf
\bibitem{ST}X. Shen, Y. Tian, \textsl{Canonical subgroups of $\pi$-divisible $\mathcal{O}$-modules}, in preparation.
\bibitem{Shin}S. W. Shin, \textsl{Galois representations arising from some compact Shimura varieties}, Ann. Math. (2), 173 (3) (2011), 1645-1741.
\bibitem{TY}R. Taylor, T. Yoshida, \textsl{Compatibility of local and global Langlands correspondences}, J. Amer. Math. Soc. 20 (2) (2007), 467-493.
\bibitem{U}E. Urban, \textsl{Eigenvarieties for reductive groups}, Ann. Math. 174 (2011), 1685-1784.
\bibitem{W} T. Wedhorn, \textsl{Ordinariness in good reductions of Shimura varieties of PEL-type}, Ann. Sci. de l'ENS 32, (1999), 575 - 618.

\end{thebibliography}
\end{document}